\documentclass{imsart}
\usepackage{mathrsfs}
\usepackage[a4paper, top=2cm, bottom=2cm, left=2cm, right=2cm]{geometry}

\RequirePackage{amsthm,amsmath,amsfonts,amssymb}
\RequirePackage[numbers]{natbib}
\RequirePackage[colorlinks,citecolor=blue,urlcolor=blue]{hyperref}
\RequirePackage{graphicx}

\startlocaldefs

\renewcommand{\P}{\mathbf{P}}

\newcommand{\E}{\mathbf{E}}

\newcommand{\R}{\mathbb{R}}
\newcommand{\N}{\mathbb{N}}
\newcommand{\F}{\mathcal{F}}

\newcommand{\bigO}{\mathcal{O}}

\theoremstyle{plain}

\newtheorem{theorem}{Theorem}[section]
\newtheorem{lemma}{Lemma}[section]
\newtheorem{remark}{Remark}[section]

\theoremstyle{remark}

\newtheorem{result}[theorem]{Result}
\DeclareMathOperator*{\essinf}{ess\,inf}

\endlocaldefs

\begin{document}

\begin{frontmatter}
\title{Large Deviations Asymptotics for Unbounded Additive Functionals of Diffusion Processes}
\runauthor{Bazhba, Blanchet, Laeven and Zwart}
\runtitle{Large deviations for additive functionals of diffusion processes}

\begin{aug}
\author[A]{\fnms{Mihail}~\snm{Bazhba}\ead[label=e1]{M.Bazhba@uva.nl}\orcid{0000-0002-0971-6326}},
\author[B]{\fnms{Jose}~\snm{Blanchet}\ead[label=e2]{jblanche@stanford.edu}\orcid{0000-0001-5895-0912}},
\author[A]{\fnms{Roger J.~A.}~\snm{Laeven}\ead[label=e3]{R.J.A.Laeven@uva.nl}\orcid{0000-0001-9582-3955}},
\and
\author[C]{\fnms{Bert}~\snm{Zwart}\ead[label=e4]{Bert.Zwart@cwi.nl}\orcid{0000-0001-9336-0096}}
\address[A]{Quantitative Economics,
University of Amsterdam, Amsterdam, Netherlands\printead[presep={,\ }]{e1,e3}}

\address[B]{Management Science and Engineering,
Stanford University, Palo Alto, United States of America\printead[presep={,\ }]{e2}}

\address[C]{Stochastics,
Centrum Wiskunde \& Informatica, Amsterdam, Netherlands\printead[presep={,\ }]{e4}}
\end{aug}

\begin{abstract}
We study large deviations asymptotics for a class of unbounded additive functionals, interpreted as normalized accumulated areas, of one-dimensional Langevin diffusions with sub-linear gradient drifts.
Our results provide parametric insights on the speed and the rate functions in terms of the growth rate of the drift and the growth rate of the additive functional.
We find a critical value in terms of these growth parameters that dictates regions of sub-linear speed for our large deviations asymptotics.
Our approach is based upon various constructions of independent interest, including a decomposition of the diffusion process in terms of alternating renewal cycles and a detailed analysis of the paths during a cycle using suitable time and spatial scales.
The key to the sub-linear behavior is a heavy-tailed large deviations phenomenon arising from the principle of a single big jump coupled with the result that at each regeneration cycle the upper-tail asymptotic behavior of the accumulated area of the diffusion process is proven to be semi-exponential (i.e., of heavy-tailed Weibull type).
\end{abstract}


\begin{keyword}[class=MSC]
\kwd[Primary ]{60F10}
\kwd[; secondary ]{60J60}
\end{keyword}

\begin{keyword}
\kwd{Large deviations}
\kwd{diffusion processes}
\kwd{heavy tails}
\end{keyword}

\end{frontmatter}

\section{Introduction}

A rich body of theory, pioneered by Donsker and Varadhan in classical work \cite{MR386024, MR428471, MR690656}, provides powerful tools designed to study large deviations for additive functionals of geometrically ergodic Markov processes.
Additive functionals can be interpreted as the ``area'' under the curve of the function drawn by the Markov process evaluated at the function in question.
We are interested in large deviations for the empirical long-term average of functionals that satisfy a polynomial growth condition.

Although a lot of progress has been made in this area, the prevailing assumptions in the literature are often not applicable to natural functionals that are unbounded.
In particular, most existing general results related to large deviations for additive functionals of diffusion and related processes assume the functional of interest to be bounded in a suitable weighted norm; see, for example, \cite{ney1987markov, 10.1214/aop/1176990736, veretennikov1994large, wu1994large, bryc1997large, dembo1997uniform, kontoyiannis2005large, MR2391248, MR3185356, den2019properties} and the references therein.
A much smaller literature (referenced below) has examined the large deviations asymptotic behavior of additive functionals over large time scales in cases where the additive functional is unbounded, typically restricting to specific Markov processes.

In this paper, we develop tail asymptotics for polynomial-growth additive functionals of one-dimensional Langevin diffusions with log-concave densities.
We study how these tail asymptotics depend on the polynomial exponent $p$ of the unbounded additive functional and the growth $\kappa \in (0,1]$ of the  diffusion's drift vector field (i.e., the derivative of the log-density).

Specifically, we assume the drift of the Langevin diffusion $X$ is of the form $x \mapsto -sgn(x)|x|^\kappa, \ \kappa \in (0,1]$, and we impose a growth condition on the additive functional $\int_{0}^{t}f(X(s))\mathrm{d}s$ given by $f(x) = |x|^p$, \  $p > 2\kappa$.
The empirical average
$
\mathcal{A}(t)= \frac{1}{t}\int_{0}^{t}f(X(s))\mathrm{d}s$,
which represents the normalized area under the paths of $\{f(X(s))\}_{s\geq 0}$, converges in probability to
$\E(|X(\infty)|^p)$, as $t \to \infty$, with $X(\infty)$ a random variable distributed as the stationary law of the diffusion process $X$.
Our main goal is to examine the behavior of the upper tail $\P(\mathcal{A}(t) > c)$, with
$c > \E(|X(\infty)|^p)$, as $t\rightarrow\infty$.
We show that as long as $p>2\kappa$, this probability decreases to zero exponentially with sub-linear speed of order $t^{(\kappa+1)/(p+1-\kappa)}$.
This is markedly different from the case $p\leq 2\kappa$, in which the speed is linear \cite{bryc1997large, wuyiao2008}.


Let us explain, on an intuitive level, how we obtain our results.
Suppose that we decompose the path of the Langevin diffusion into ``cycles'' corresponding to returns to a neighborhood of the origin.
We will show that the area under each cycle is heavy-tailed, and that standard intuition from heavy tails applies, i.e., a single big value of the area under a cycle dominates the value of $\mathcal{A}(t)$.
To show that the area for a particular cycle is heavy-tailed, we consider the probability that the area under the curve within a cycle is of order $t$, and examine the diffusion within a cycle on time scales of order $O(t^\beta)$ and spatial scales of order $O(t^{\alpha/p})$ with $\beta\triangleq (1-\kappa)/(p+1-\kappa)$ and $\alpha\triangleq p/(p+1-\kappa)$.
The resulting process,  $X_t(u)\triangleq\frac{1}{t^{\alpha/p}}X(ut^\beta)$, $u\geq 0$, is a diffusion process
with variance parameter of order $O(t^{-\gamma})$ with $\gamma \triangleq (1+\kappa)/(p+1-\kappa)$.
Since $\gamma>0$, the noise of this scaled diffusion process is small.
In terms of this scaled small-noise diffusion, the area under the curve within the cycle in question must be of order $O(1)$, which is shown to be governed by a large deviations rate with speed $O(t^{\gamma})$.
Under the assumption $p > 2\kappa$, $\gamma$ is smaller than $1$.
Hence, we conclude that a typical cycle exhibits semi-exponential tails with shape parameter $\gamma$ (i.e., tails of the form $\exp\{-t^\gamma\}$) and therefore is heavy-tailed.
This stands in contrast to the (more conventional) linear speed typically found in large deviations analysis.

Since the number of cycles in $t$ units of time is approximately linear and obeys a large deviations principle with a speed that is linear in $t$ (i.e., light tails), the overall large deviations asymptotics of the unbounded additive functional of interest displays similar behavior to that of $O(t)$ i.i.d.~increments with semi-exponential tails with shape parameter $\gamma$.
This explains both the rate and the qualitative behavior of how large deviations occur in our setting, consistent with the mechanism in which heavy-tailed large deviations occur.
That is, the most likely way in which large deviations occur in our current setting involves a single cycle lasting $O(t^\beta)$ units of time, and exhibiting fluctuations of order $O(t^{\alpha/p})$.
This is sufficient to accumulate a contribution of order $O(t)$, which is enough to generate a large deviations event of order $O(1)$ away from the typical behavior in the empirical average of the additive functional.

The mathematical details of our analysis are intricate.
A first major obstacle that complicates our analysis is the implementation of the Freidlin-Wentzell sample-path large deviations principle (LDP), which is applicable only for small-noise diffusions with drift functions that are Lipschitz continuous.
In our case, these conditions are not satisfied and their violation requires a delicate analysis.
Deploying the temporal and spatial scales described above, we study the tail asymptotics of $\mathcal{A}(t)$, as $t \ \to \infty$, through a resulting small-noise diffusion.
To use the Freidlin-Wentzell LDP, we next construct suitable auxiliary diffusions satisfying the Freidlin-Wentzell framework, and use the areas of these auxiliary diffusions as upper and lower bounds to the area of the small-noise diffusion.
This approach enables us to eventually establish upper and lower large deviations bounds, which we prove to coincide.

While we overcome the issues originating from the violation of the Freidlin-Wentzell LDP assumptions, a second major obstacle in developing tail asymptotics per regeneration cycle remains in that the area functional $\int_{0}^{\mathcal{T}(\xi)}|\xi(s)|^p \mathrm{d}s$, $\mathcal{T}(\xi)= \inf\{ t \geq 0: \xi(t)=0\}$, with $\xi$ an absolutely continuous function with square integrable derivative, is not continuous in the uniform topology.
Although the area functional is discontinuous in general, it is continuous over time horizons that are deterministic (as opposed to first passage times) and, hence, we approximate the area over the respective regeneration cycle, which has a random endpoint, by the area over a sufficiently large time horizon.
Then, the large deviations upper bound is obtained as the optimal value of a variational problem, which we show to converge to a certain value as the time horizon tends to infinity.
For the lower bound, we confine the area under the curve of the process under consideration by intersecting it with events that track the trajectory of the most probable path.
This leads to a variational problem for which the optimal value is shown to converge to the optimal value of the variational problem associated with the upper bound.

Although they occur naturally, large deviations for unbounded additive functionals of Markov processes have been analyzed in only a limited number of papers,
in two strands of the literature.
First, in the probability theory literature, for the specific case of Markov random walks with light-tailed increments, we refer to \cite{gulinskii1994large, duffy2010most, blanchet2013large, duffy2014large, bazhba2020sample}, whereas \cite{guillemin1998area, borovkov2003integral, kulik2005tail} establish results for the areas under the workload process and under the queuing process in a single server queue over a regeneration cycle.
Second, in the physics literature,
\cite{nickelsen2018anomalous, nickelsen2022noise, smith2022anomalous} recently analyzed the asymptotic behavior of additive functionals of the Gaussian Ornstein-Uhlenbeck process and obtained large deviations estimates with sub-linear speed for functionals of polynomial growth faster than quadratic.
Extensions of the results in \cite{nickelsen2018anomalous} to the case of stationary Gaussian processes are in \cite{meerson2019anomalous}.
The results in these papers, all focused on Gaussian settings, provide physical insights largely based on physical arguments.
In this work, we unravel the probabilistic mechanisms that generate the heavy-tailed behavior of the additive functionals under consideration.
Our approach enables us to rigorously prove upper-tail large deviations asymptotics for additive functionals of diffusions with sub-linear drift functions, explicitly compute the speed and rate function of the large deviations asymptotics, and present a methodology of independent interest to deal with more general settings and cases.

The remainder of the paper is organized as follows:
In Sections~\ref{model_presentation} and~\ref{main_result_presentation}, we present the model and our main result (Theorem~\ref{asympotics-for-cal-A}), namely, the upper large deviations asymptotics for the empirical average of the unbounded additive functionals under consideration.
In Section~\ref{overall-methodology}, we present our methodology including the cycle decomposition and scaling techniques mentioned above, which provides a heuristic proof of our main result and a road map of the proof techniques that we use to make the analysis rigorous.
In Section~\ref{2-main-building-blocks}, we present a key lemma (Lemma~\ref{semi-exponential-property-cycles}) for the proof of the main result and, based on this, conclude its proof.
In Section~\ref{continuity-properties-V}, we present and prove some useful continuity properties of the variational problems associated with Lemma~\ref{semi-exponential-property-cycles}.
Section~\ref{proof-lemma-21} is devoted to the proof of Lemma~\ref{semi-exponential-property-cycles}, and  Section~\ref{section-proof-of-ndelta-asy} contains the proof of auxiliary Lemma~\ref{bounded-moment-generating-function-t-1-delta}.
Finally, the paper contains two appendices covering existing large deviations results tailored to our setting and results on upper bounds for densities of first passage times of diffusion processes.

\section{Model Description and Main Result}

\subsection{Model}\label{model_presentation}

We consider a time-homogeneous diffusion process $X$ which is induced by the ensuing stochastic differential equation:
\begin{align}\label{eq:original-multid-stoc-difeq}
\mathrm{d}X(t)=\mathbf{D}(X(t))\mathrm{d}t+ 2^{-1/2}\sigma \mathrm{d}B(t), \qquad t \geq 0.
\end{align}
Here, $B(t)$ denotes a standard Brownian motion and $\sigma$, the diffusion coefficient, is a positive constant.
We assume that $X(0)=0$
and we impose a sub-linear growth condition on the drift, as follows:
$$\mathbf{D}(x)=-sgn(x) |x|^{\kappa}, \qquad\kappa \in (0,1].$$
The stochastic differential equation in~\eqref{eq:original-multid-stoc-difeq} is properly defined:
since the drift coefficient is a locally bounded measurable function and the diffusion coefficient is a constant, in view of Theorem~10.1.3 in \cite{stroock1997multidimensional}, Eqn.~\eqref{eq:original-multid-stoc-difeq} has a unique weak solution; see also Proposition~5.3.6 of \cite{karatzas1998brownian}.

The purpose of this paper is to prove upper large deviations asymptotics for the normalized accumulated area $\mathcal{A}(t)$, as $t$ tends to $\infty$, where
\begin{equation*}
\mathcal{A}(t) \triangleq \frac{1}{t}\int_{0}^{t}f(X(s))\mathrm{d}s,  \qquad f(x)=|x|^p,
\end{equation*}
and $p > 0$ is such that $p>2\kappa$. Throughout this paper, we frequently use the following notation:
\begin{equation}
\label{eq-abc}
    \alpha \triangleq   \frac{p}{p+1-\kappa}, \qquad \beta\triangleq \frac{1-\kappa}{p+1-\kappa}, \qquad \gamma \triangleq \frac{1+\kappa}{p+1-\kappa}.
\end{equation}

\subsection{Main Result}\label{main_result_presentation}

In this subsection, we present our main result.
We start with a few definitions that are used in the statement of our main result.
Let $\mathbb{C}[0,T]$, $T>0$, be the space of continuous functions over the domain $[0,T]$ and let $\mathscr{H}[0,T]$ be the subspace of $\mathbb{C}[0,T]$ that contains absolutely continuous functions with square integrable derivative.
Furthermore, for a continuous function $g: \R \rightarrow \R$ and $x_0 \in \R$,
let 
\begin{align}
	I_{x_0, g}^{T}(\xi)
	\triangleq
	\begin{cases}
	\int_{0}^{T}\frac{|\dot{\xi}(s)-g(\xi(s))|^2 }{\sigma^2}\mathrm{d}s, & \ \xi \in \mathscr{H}[0,T] \ \& \ \xi(0)=x_0
	\\
	\infty, & \ \text{otherwise}
	\end{cases},
\label{eq:rate}	
\end{align}
using the dot notation for (time) derivatives, and define for $m>0$, \begin{equation}
\label{eq-regionvariationalproblem}
A(x_0,T,m) \triangleq \left\{\xi \in \mathbb{C}[0,T]: \int_{0}^{T}|\xi(s)|^p\mathrm{d}s \geq m, \ \xi(0)=x_0\right\},
\end{equation}
and the variational problem
\begin{equation}
\label{eq-variationalproblem}
    \mathrm{V}(x_0,T,g,m)\triangleq \inf_{\xi \in A(x_0,T,m)} I_{x_0, g}^T(\xi).
\end{equation}
We say that $\xi$ is a feasible path for the variational problem $\mathrm{V}(x_0,T,g,m)$ if $ \xi \in A(x_0,T,m)$.
Finally, let $X(\infty)$ be a random variable distributed as the stationary law of the diffusion process $X$,
the existence of which is ensured due to e.g., \cite{MR1730651}.
Now we are equipped to state our main result. 

\begin{theorem}\label{asympotics-for-cal-A}
For every fixed $p> 2\kappa$ and $b >\E(|X(\infty)|^{p})$,
	\begin{equation}\lim_{t \to \infty}\frac{1}{t^{\gamma}}\log \P\left(\mathcal{A}(t) \geq b \right)=-\big(b-\E(|X(\infty)|^{p})\big)^{\gamma} \cdot  \mathrm{V}(0,\infty,\mathbf{D},1).
	\end{equation}
\end{theorem}

We note that the above theorem fully characterizes the most likely scenario leading to large deviations through the associated variational problem.
The typical path of $\mathcal{A}$ can be estimated by the Law of Large Numbers.
For Harris recurrent Markov processes, it is known that the empirical average $\frac{1}{t}\int_{0}^{t}f(X(s)) \mathrm{d}s$ converges in probability to $\E(f(X(\infty)))$, as $t \to \infty$.
For $f(x)=|x|^p$, $\E(|X(\infty)|^{p})$ can be computed by the stationary distribution $\pi$ of $X$: $\E(|X(\infty)|^{p}) = \int_{\R}|s|^p\pi(\mathrm{d}s)$. 
Consequently, the Law of Large Numbers dictates that the typical value of $\mathcal{A}(t)$ is  $\E(|X(\infty)|^{p})$. 
The rate function associated with the large deviations asymptotics indicates that the most likely large deviations behavior away from the typical path occurs due to a single large cycle.

\begin{remark}
For the lower tail, large deviations asymptotics  hold more conventionally (i.e., with linear speed). To substantiate this: Fix $\epsilon >0$, and
choose $M$ large enough so that $\E(\min\{M, |X(\infty)|^p\}) > \E(|X(\infty)|^p)-\epsilon $.  Note that such $M$ exists in view of the monotone convergence theorem.
Then, $\P\left( \int_{0}^{t}|X(s)|^p  \mathrm{d}s \leq (\E(|X(\infty)|^p)-\epsilon)t\right) \leq \P\left( \int_{0}^{t} \min\{M, |X(s)|^p\}  \mathrm{d}s \leq (\E(|X(\infty)|^p)-\epsilon)t\right)$. Finally, since  $\min\{M, |X(\cdot)|^p\}$ is uniformly bounded, the Donsker-Varadhan theory (\cite{MR386024}) implies a large deviations upper bound with linear speed.
\end{remark}

\subsection{Methodology}\label{overall-methodology}

Our proof strategy relies on a suitable decomposition of $X$ into cycles using regenerative analysis of diffusions.
The definition of cycles involves the introduction of a small boundary layer which is used to define ``entry'' points to the origin:
we introduce a decomposition which is used to account for the contribution of paths on excursion that ``exit'' and re-enter through the boundary layer within a single cycle.
Towards this purpose, we define the regeneration cycles
by considering the following hitting times for $X(\cdot)$:
$B_0^\delta =0$,
$A_1^\delta = \inf\{ t \geq 0: |X(t)| = \delta\}$,
$B_1^\delta = \inf\{t > A_1^{\delta} : |X(t)| = 0 \}, \ldots, A_{n+1}^\delta = \inf\{ t>B_n^\delta: |X(t)| = \delta\}$, and
$B_{n+1}^\delta = \inf\{t > A_{n+1}^{\delta} : |X(t)| = 0 \}$, where $\delta>0$.	
The diffusion $X$ is regenerative with respect to the sequence $\big\{B^{\delta}_j\big\}_{j \geq 0}$ which induces a renewal process $N^\delta(\cdot)$;
	$
	N^{\delta}(t) \triangleq \max\left\{k \geq 0: B^{\delta}_k \leq t\right\}
	$.
This enables us to decompose the process $\mathcal{A}(t)$ as follows:

	\begin{align*}
	\mathcal{A}(t)=\frac{1}{t}\Bigg[
	\underbrace{\int_{0}^{B_1^\delta} |X(s)|^p \mathrm{d}s}_{ \triangleq C_1^{\delta}}
	+
	\sum_{j=1}^{N^{\delta}(t)-1}\Bigg(\underbrace{\int_{B_j^{\delta}}^{B_{j+1}^\delta} |X(s)|^p \mathrm{d}s}_{\triangleq C_{j+1}^{\delta}}  \Bigg)
	+  {\int_{B^\delta_{N^{\delta}(t)}}^{t} |X(s)|^p \mathrm{d}s}
	\Bigg].
	\end{align*}

We introduce some notation employed later in the analysis of $\mathcal{A}$.
Define $\tau_j^{\delta} = B^{\delta}_{j}-B^{\delta}_{j-1}$ as the inter-arrival times of the renewal process $N^{\delta}$ and, for each $j=1,\ldots,N^{\delta}(t)$, denote by $C_{j}^{\delta}=\int_{B^{\delta}_{j-1}}^{B^{\delta}_{j}}|X(s)|^p\mathrm{d}s$ the area under the trajectory of $X$ during $[{B^{\delta}_{j-1},B^{\delta}_{j}}]$.
A pivotal step towards deriving our main result (Theorem~\ref{asympotics-for-cal-A}) consists of establishing that each cycle $C_j^{\delta}$ exhibits heavy-tailed behavior---in particular, semi-exponential behavior---in the asymptotic regime.
Here, we briefly describe the main argument leading to the semi-exponential property of the $C_j^\delta$'s.

First,
note that $X$ satisfies the integral equation $X(w)=\int_{0}^{w}-sgn(X(s))|X(s)|^{\kappa}\mathrm{d}s+ 2^{-1/2}\sigma B(w)$, $X(0)=0$.
Next, we use the space scaling $t^{\alpha/p}$ and the time scaling $t^{\beta}$. 
Then, we have that
	\begin{align}
	X_t(u)\triangleq
	\frac{1}{t^{\alpha/p}}X(ut^\beta)
	 &=\frac{1}{t^{\alpha/p}}\int_{0}^{ut^{\beta}}-|X(s)|^{\kappa}sgn(X(s))\mathrm{d}s
	+
	\frac{2^{-1/2}\sigma}{t^{\alpha/p}}B(ut^{\beta})  	
	\nonumber \\
	&
	=	-\int_{0}^{u}\frac{t^{\beta}}{t^{\alpha/p}}|X(st^{\beta})|^{\kappa}sgn(X(st^{\beta}))\mathrm{d}s
	+
	\frac{2^{-1/2}\sigma}{t^{\alpha/p}}B(ut^{\beta})  	\nonumber
	\\
	&
	=
	-\int_{0}^{u}\left|\frac{X(st^{\beta})}{t^{(\alpha/p-\beta)\frac{1}{\kappa}}}\right|^{\kappa}sgn\left(\frac{X(st^{\beta})}{t^{\alpha/p-\beta}}\right)\mathrm{d}s
	+
	\frac{2^{-1/2}\sigma}{t^{\alpha/p-\beta/2}}B(u) \nonumber
	\\
	&
=-\int_{0}^{u}\left|X_t(s)\right|^{\kappa}sgn\left(X_t(s)\right)\mathrm{d}s
	+
	\frac{2^{-1/2}\sigma}{t^{\alpha/p-\beta/2}}B(u), \label{eq-representation-Xt}
	\end{align}
in distribution.
To employ tools from the arsenal of large deviations theory, we work with the small-noise diffusion process $X_t(u)$, which satisfies the following SDE:
\begin{equation}
\label{eq-Xts-definition}
	\mathrm{d}X_t(s)=-sgn(X_t(s))|X_t(s)|^{\kappa}\mathrm{d}s+ \frac{1}{t^{\gamma/2}} 2^{-1/2} \sigma \mathrm{d}B(s).
\end{equation}

The tail asymptotics of $C_1^\delta$ are derived 
by relating the distribution of $\frac 1 t C_1^\delta$ to that of an image of $X_t$ obtained by applying the area functional
$\F(\xi) \triangleq \int_{0}^{\mathcal{T}(\xi)}|\xi(s)|^p \mathrm{d}s$, $\mathcal{T}(\xi)= \inf\{ t \geq 0: \xi(t)=0\}$, to it; see Lemma~\ref{distributional-equalities-via-scaled-diffusion-process}, item $i$).
Drawing intuition from classical large deviations theory, large deviations rates are expected to hold with speed of order $t^\gamma$,
where $\gamma<1$ when $p>2\kappa$.
However, the Freidlin-Wentzell LDP does not apply straightforwardly to the small-noise diffusion $X_t$, since its drift coefficient is not a Lipschitz continuous function and the area functional $\cal{F}$ is not suitably continuous.
Whereas this does not lead to a different LDP than one would expect from Freidlin-Wentzell theory,
these obstacles make our proof for the tail asymptotics of $C_1^\delta$ substantially more involved.
For more information regarding our approach, we refer to Lemma~\ref{semi-exponential-property-cycles} and its accompanying discussion.

Following the strategy above, we establish the semi-exponential property of the area under the curve drawn by the diffusion process within a single cycle (Lemma~\ref{semi-exponential-property-cycles}), and we can proceed to deal with the random number of terms in the decomposition of $\mathcal{A}$ using the typical value of $N^\delta$ (see Lemma~\ref{bounded-moment-generating-function-t-1-delta}).
Next, the principle of one big jump for sums of Weibull-tailed random variables applies (see \cite{gantert2014large}), and hence, we establish both the rate and the qualitative behavior of how large deviations occur in this setting;
that is,
the most likely way large deviations occur is through a single large cycle.
In Figure~\ref{proof-flow-chart}, we graphically illustrate the connections between the various lemmas and results leading to the proof of our main result (Theorem~\ref{asympotics-for-cal-A}).
\begin{figure}[h]
\centering
\vskip -0.452cm
\includegraphics[scale=0.56]{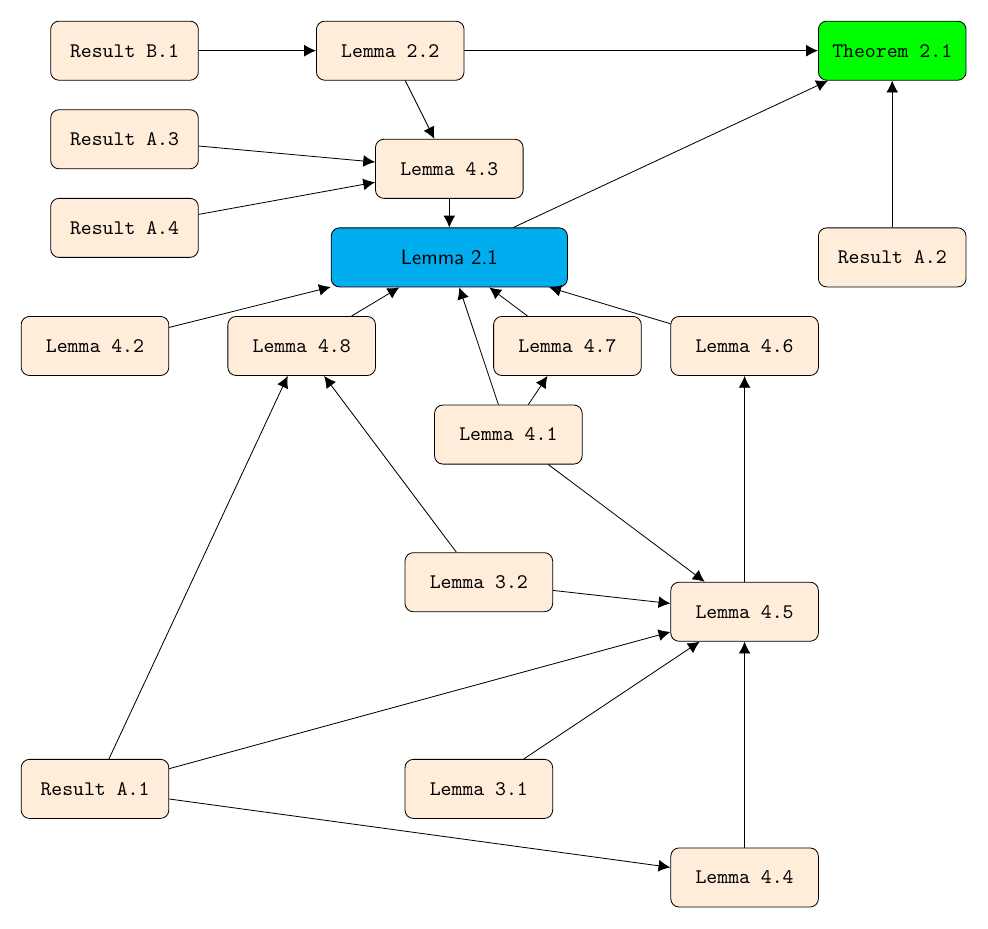}
\caption{The figure displays a diagram of our proof.
The main result is colored green and the  key intermediate result used in its proof is colored blue.}
\label{proof-flow-chart}
\end{figure}


We believe that our methodology is powerful enough so that it can be applied to more general diffusion processes.
To illustrate the broader applicability of our methodology, consider the case where the drift function is super-linear ($\kappa > 1$), and recall our decomposition of the path of the Langevin diffusion into cycles.
Next, employing time scales of order $O(t^\beta)$ and spatial scales of order $O(t^{\alpha/p})$, 
we obtain a  small-noise diffusion with diffusion coefficient of order $O(t^{-\gamma/2})$.
Then, at each cycle, we consider the probability that the area swept under the curve of the additive functional of the small-noise diffusion is of order $O(1)$.
Following a similar approach as in the case of $\kappa \leq 1$, the large deviations rates would be of order $O(t^{\gamma})$.
Under the assumption $p > 2\kappa$, $\gamma$ is smaller than $1$, thus the Weibull-like behavior of each cycle is obtained.
Now, however, we have $\beta<0$, inducing short cycles, due to the drift's strong pull towards the origin.
Hence, for the case $\kappa>1$, the overall large deviations asymptotics of the additive functional of interest is similar to that of $O(t)$ i.i.d. increments with semi-exponential tails (i.e., Weibull-type heavy tails) with shape parameter $\gamma$.
Consequently, the principle of one big jump is also expected to apply, and large deviations occur due to a single large cycle.
We note furthermore that very similar splitting and scaling arguments can be applied for regularly varying drift functions satisfying $x\mapsto-sgn(x)|x|^\kappa \cdot l(x)$, where $l(\cdot)$ is a non-negative slowly varying function.
(The same type of generalization can be allowed similarly for the function $f(\cdot)$.)
This change would introduce a slowly varying component in the asymptotic result at the expense of making the development significantly more involved without adding much more conceptual insight; thus, we have decided to simplify the exposition under our current assumptions.
We leave the details of such generalizations, each requiring subtle and elaborate individual treatment, for future work. 

\subsection{A Key Lemma and the Proof of the Main Result}\label{2-main-building-blocks}

In this subsection, we state a key result---Lemma~\ref{semi-exponential-property-cycles}
---used subsequently in the proof of our main result---Theorem~\ref{asympotics-for-cal-A}.
We start with a detailed road map of the proof of the tail asymptotics for $C_1^\delta$ established in Lemma~\ref{semi-exponential-property-cycles}; its full proof is contained in Sections~\ref{continuity-properties-V} and~\ref{proof-lemma-21}.

\subsubsection{Tail asymptotics for $C_1^{\delta}$}
The tail asymptotics for the area $C_1^\delta$ are provided by the following lemma, which plays a central role in the paper. 	
Recall that $A_1^\delta = \inf\{ t \geq 0: |X(t)| = \delta\}$, $B_1^\delta = \inf\{t > A_1^{\delta} : |X(t)| = 0 \}$ and $C_1^\delta=\int_{0}^{B_1^{\delta}}|X(s)|^p \mathrm{d}s$.

\begin{lemma}\label{semi-exponential-property-cycles}
For every fixed $p > 2\kappa$, $b \geq 0$ and $\delta>0$,
	\begin{equation}
	\lim_{t \to \infty}\frac{1}{t^{\gamma}}\log\P\left(C_1^\delta \geq bt\right)=-b^{\gamma}\cdot \mathrm{V}(0,\infty,\mathbf{D},1).
	\end{equation}
\end{lemma}
We prove this result for $b\equiv 1$, the more general case trivially follows by replacing $t$ with $bt$.
The proof of Lemma \ref{semi-exponential-property-cycles} in principle aims to invoke the LDP Result~\ref{FW-SPLDP} in the Appendix, combined with sample-path based comparison arguments.
Using the space scaling $t^{\alpha/p}$ and the time scaling $t^{\beta}$, 
we obtain the small-noise diffusion $X_t(\cdot)=X(\cdot\  t^\beta)/t^{\alpha/p}$ over large time-scales (see Lemma~\ref{distributional-equalities-via-scaled-diffusion-process}).
Then, $\frac 1 t C_1^{\delta}$ can be viewed as the image of the scaled diffusion process $X_t$ to which we apply the area functional $\F(\xi) = \int_{0}^{\mathcal{T}(\xi)}|\xi(s)|^p \mathrm{d}s$, $\mathcal{T}(\xi)= \inf\{ t \geq 0: \xi(t)=0\}$.

However, the sample-path LDP in Result~\ref{FW-SPLDP} holds with respect to the uniform topology, and one of the main issues that complicates our analysis is the discontinuity of $\F$ under the uniform metric.
Consequently, the proof for the tail asymptotics of $\frac{1}{t}C_1^\delta$ gets significantly more intricate than e.g., simply applying the contraction principle.
We deal with this issue by a direct approach, which involves the derivation of large deviations upper and lower bounds and proving that they coincide.
This is largely the contents of Section~\ref{proof-lemma-21}.

For the upper bound,
we note that Result~\ref{FW-SPLDP} does not apply to diffusion processes with non-Lipschitz continuous drift.
To overcome this issue,
we construct a diffusion process ${X}_{t,\epsilon}$ with a modified, Lipschitz continuous drift $u_\epsilon$ such that $u_\epsilon \geq \mathbf{D}$.
Hence, in view of stochastic dominance results for diffusion processes (see \cite{yamada1973comparison}),
${X}_{t,\epsilon}$  serves as a stochastic upper bound for $X_t$, and the area swept by ${X}_{t,\epsilon}$ upper-bounds the area swept by $X_t$ (see Lemma~\ref{stochastic-upper-bound-via-modified-dif-n-iid-cjs}).
We note that the area under the trajectories of $X_{t,\epsilon}$ can be written as a functional of $X_{t,\epsilon}$ through the map $\mathcal{F}$.
To deal with the discontinuity of the functional $\mathcal{F}$, it is tempting to directly truncate the first passage time to $0$ with a sufficiently large value of $T$.
However, this approach, when applied directly, fails.
This can already be seen in the case $\kappa=1$, the reason being that the drift vanishes as the process approaches zero, making the time to hit zero relatively large.
Nevertheless, the contribution to the area swept by the process near zero is also small.
Hence, we introduce a boundary layer around the origin to account for the area swept by the process near zero by means of a reflected diffusion.
This construction is studied in Lemma~\ref{area-for-epsilon-bounded-trajectories}, which ultimately shows that the contribution of the area swept by the process around the origin is negligible if the scaled process starts within $O(t^{-\alpha/p})$ units from the origin---this initial condition turns out to be important in the analysis.
Once this contribution is removed, we can focus on the area accumulated only outside the boundary layer.
In this case, the key quantity involves the first passage time to a strictly positive level (uniformly bounded away from the origin).
The drift remains bounded away from zero, and the truncation strategy for a sufficiently large value of a time horizon $T$ now pays off (see Lemma~\ref{high-occupation-interval-asymptotics}).
This overall construction enables us to express the area of the modified diffusion ${X}_{t,\epsilon}$ as a functional over the finite time horizon $[0,T]$.
This functional is known to be continuous with respect to the uniform metric.
For $T$ large enough, we show that the absolute area of the scaled diffusion process $X_{t,\epsilon}$ over the time horizon $[0,T]$ serves as an asymptotic upper bound for $\frac{1}{t}C_1^{\delta}$.
Finally, by Result~\ref{FW-SPLDP} and the contraction principle, we obtain the large deviations upper bound (Lemmas~\ref{tail-asymptotics-for-semicycle} and~\ref{log-asy-sum-of-eps-semi-cycles}).

For the lower bound,
we construct a diffusion process $\tilde{X}_{t,\epsilon}$ with a modified  drift $\tilde{u}_\epsilon$.
Similar as for the upper bound, we construct  $\tilde{X}_{t,\epsilon}$
 in such a way that its drift is Lipschitzian and the area swept by $\tilde{X}_{t,\epsilon}$ now lower-bounds the area swept by $X_t$.
 Then, 
 we confine the area swept under the modified diffusion process over a suitable fixed time horizon (Lemma~\ref{lower-bound-with-modified-process}).
 Subsequently, we derive a variational problem associated with the lower bound.
 Finally, we prove that the variational problem related to the lower bound has the same optimal value as the variational problem associated with the large deviations upper bound (Lemma~\ref{logarithim-lower-bound-for-full-cycle}), see also Figure~\ref{proof-flow-chart} for a graphical overview of the proof structure.

\subsubsection{Tail asymptotics for $N^{\delta}$}
Before we give the proof of our main result, we provide a final preparatory lemma, which implies that large deviations of
%
%
the renewal process associated with the number of visits to $0$, $N^{\delta}(t)=\max\left\{k \geq 0: B^{\delta}_k \leq t\right\}$, occur at linear speed.
A necessary component of our analysis is the moment generating function (m.g.f.) of $\tau_1^\delta=B_1^\delta$, which we prove to be finite in a neighborhood around $0$.

	\begin{lemma}\label{bounded-moment-generating-function-t-1-delta}
		It holds that $\sup\left\{\theta: \E\left(e^{\theta \tau_1^\delta}\right) <\infty\right\}>0$.
  Consequently, $\P\left(\left|\frac{N^{\delta}(t)}{t} -\frac{1}{\E(\tau^{\delta}_1)}\right| \geq  x \right) = O(e^{-t\tilde{I}_N(x)})$
  for some $\tilde{I}_N(x)>0$, $x>0$, as $t\rightarrow\infty$.
	\end{lemma}
The second part of the lemma trivially follows from the first part by Cram\'er's Theorem.
We show the existence of $\E(e^{\theta \tau_1^\delta})$ by invoking upper bounds for first passage time densities of diffusion processes; see Result~\ref{bound-hitting-times-of-diffusion-processes}.
Since Result~\ref{bound-hitting-times-of-diffusion-processes} is not directly applicable to our case, that is, the conditions of the result are not satisfied by $\tau_1^\delta$, we first construct suitable auxiliary first passage times which serve as stochastic upper bounds for $\tau_1^\delta$ and we prove that their m.g.f.'s are finite; see Section~\ref{section-proof-of-ndelta-asy}.

\subsubsection{Proof of the main result}
We conclude this section with the proof of our main result (Theorem~\ref{asympotics-for-cal-A}).

\begin{proof}[Proof of Theorem~\ref{asympotics-for-cal-A}]
 Recall that, by Lemma~\ref{semi-exponential-property-cycles}, the $C_i^\delta$'s exhibit an asymptotic semi-exponential behavior.
 This entails that LDP Result~\ref{LDP-semiexponential-rvs} is applicable in our case.
 Moreover, applying standard results for regenerative processes (see, for example, Theorem~1.2 in \cite{asmussen2008applied}),
 we obtain
 \begin{equation}\label{equivalence-Epi-ratio-ECoEB}
 \E_{\pi}(|X(1)|^p)=\E(|X(\infty)|^p)=\frac{1}{\E(Y)}\E_0 \left(\int_{0}^{Y}|X(s)|^p \mathrm{d}s\right),
 \end{equation}
 where $Y$ is any nonlattice cycle length r.v. with finite mean, and recall that $\pi$ is the steady state distribution of $X$.

Let $\epsilon >0$ be fixed.
Then,
\begin{align*}
& \P\left(\mathcal{A}(t) \geq b \right)
\leq
\P\left(\sum_{i=1}^{N^{\delta}(t)+1}C_i^{\delta} \geq b t\right)
\leq
\P\left(\sum_{i=1}^{\left\lceil t(\frac{1}{\E(\tau^{\delta}_1)}+\epsilon) \right\rceil}C_i^{\delta} \geq b t\right)+ \underbrace{\P\left(N^{\delta}(t) \geq  t\left(\frac{1}{\E(\tau^{\delta}_1)}+\epsilon\right)\right)}_{\triangleq (i)}.
\end{align*}
We can apply this inequality to obtain
\begin{align*}
\limsup_{t \to \infty}\frac{1}{t^\gamma}\log \P\left(\mathcal{A}(t) \geq b \right)
&\leq
\limsup_{t \to \infty}\frac{1}{t^\gamma}\log\P\left(\sum_{i=1}^{\left\lceil t\left(\frac{1}{\E(\tau^{\delta}_1)}+\epsilon\right) \right\rceil}C_i^{\delta} \geq b t\right)\\
&
\leq
-\mathrm{V}(0,\infty,\mathbf{D},1) \left(\left(b-\frac{\E(C_1^{\delta})}{\E(\tau^{\delta}_1)}-\E(C_1^{\delta})\epsilon\right)^+\right)^{\gamma}\\
&=
-\mathrm{V}(0,\infty,\mathbf{D},1) \left(\left(b-\E(|X(\infty)|^p)-\E(C_1^{\delta})\epsilon\right)^+ \right)^{\gamma},
\end{align*}
where in the first inequality we use the principle of the largest term and the fact that $(i)$ has light-tailed asymptotics by Lemma~\ref{bounded-moment-generating-function-t-1-delta}, in the second inequality we use Lemma~\ref{semi-exponential-property-cycles} as well as Result~\ref{LDP-semiexponential-rvs},
and in the last equality we invoke \eqref{equivalence-Epi-ratio-ECoEB}.
The desired upper bound is obtained by letting $\epsilon \downarrow 0$.
For the lower bound,
\begin{align*}
&
\P\left(\mathcal{A}(t) \geq b \right)
 \geq
\P\left(\sum_{i=1}^{N^{\delta}(t)}C_i^{\delta} > b t\right)
\geq \P\left(\sum_{i=1}^{N^{\delta}(t)}C_i^{\delta} > b t, \ N^{\delta}(t) \geq t\left(\frac{1}{\E(\tau^{\delta}_1)}-\epsilon\right)
\right)
\\
&
\geq
{\P\left(\sum_{i=1}^{\left\lfloor t(\frac{1}{\E(\tau^{\delta}_1)}-\epsilon) \right\rfloor}C_i^{\delta} > b t\right)}
- \underbrace{\P\left(N^{\delta}(t) \leq t\left(\frac{1}{\E(\tau^{\delta}_1)}-\epsilon\right)\right)}_{\triangleq (ii)}.
\end{align*}
Using Lemma~\ref{semi-exponential-property-cycles} and Result~\ref{LDP-semiexponential-rvs},
\begin{align*}
\liminf_{t \to \infty}\frac{1}{t^\gamma}\log\P\left(\sum_{i=1}^{\left\lfloor t \left(\frac{1}{\E(\tau^{\delta}_1})-\epsilon\right) \right\rfloor}C_i^{\delta} > b t\right)
\geq
-\mathrm{V}(0,\infty,\mathbf{D},1) \left(b-\E(|X(\infty)|^p)+\E(C_1^{\delta})\epsilon\right)^{\gamma}.
\end{align*}
Due to Lemma~\ref{bounded-moment-generating-function-t-1-delta},
$(ii)$ has light-tailed asymptotics.
Therefore,
\begin{align*}
\liminf_{t \to \infty}\frac{1}{t^\gamma}\log\P\left(\mathcal{A}(t) \geq b t\right)
\geq
-\mathrm{V}(0,\infty,\mathbf{D},1) \left(b-\E(|X(\infty)|^p)+\E(C_1^{\delta})\epsilon\right)^{\gamma}.
\end{align*}
We obtain the desired lower bound by letting $\epsilon$ tend to $0$, which completes the proof.
\end{proof}

\section{Variational Properties of $\mathrm{V}(\cdot)$}\label{continuity-properties-V}
In this section, we prove some important properties for the variational problems that are associated with the large deviations upper and lower bounds.
Recall that $\mathbb{C}[0,T]$ is the space of continuous functions over the domain $[0,T]$, that $\mathscr{H}[0,T]$ is its subspace of absolutely continuous functions with square integrable derivative, and that $I_{x_0, g}^{T}(\xi)$ is the rate function in \eqref{eq:rate}.
We say that a function $g: \R \rightarrow \R$ is symmetric around $0$ if $g(x)=-g(-x)$.
Recall that $\mathbf{D}(x)=-sgn(x)|x|^{\kappa}, \kappa \in (0,1]$, and note, in the light of the following lemma, that $\mathbf{D}$ is symmetric around 0.
Furthermore, for $\epsilon>0$, consider the following functions:
\begin{equation}
\label{eq-modifieddrift}
u_{\epsilon}(x)
\triangleq
\begin{cases}
-sgn(x)|x|^{\kappa}, &  \ \text{for} \ |x| \geq \epsilon,
\\
-\frac{x}{\epsilon^{1-\kappa}}, &  \ \text{for} \ |x| \leq \epsilon,
\end{cases}
\qquad\qquad
\tilde{u}_{\epsilon}(x)
\triangleq
\begin{cases}
-x^{\kappa}, &  \ \text{for} \ x \geq \epsilon,
\\
-\epsilon^\kappa, &  \ \text{for} \ x < \epsilon.
\end{cases}
\end{equation}

Now define
\begin{equation}
 A_+(x_0,T,m)\triangleq\left\{\xi \in \mathbb{C}[0,T]: \int_{0}^{T}|\xi(s)|^p\mathrm{d}s \geq m, \ \xi \geq 0,  \ \xi(0)=x_0\right\},
 \end{equation}
 and let 
$\mathrm{V}_+(x_0,T,g,m)\triangleq\inf_{\xi \in A_+(x_0,T,m)}I_{x_0,g}^T(\xi)$.
\begin{lemma}
\label{continuity-properties-of-V}
\begin{itemize}
\item[$i$)] For any continuous function $g: \R \rightarrow \R$ symmetric around 0, it holds that
\[
\mathrm{V}(x_0,T,g,m)
=
\mathrm{V}_+(|x_0|,T,g,m).
\]


\item[$ii$)]
We have that
\begin{align*}
\lim_{T \to \infty}\mathrm{V}_+(|x_0|,T,u_\epsilon,m)&= \mathrm{V}_+(|x_0|,\infty,u_{\epsilon},m), \quad \text{and}\\
 \lim_{T \to \infty}\mathrm{V}(x_0,T,u_\epsilon,m)&= \mathrm{V}(x_0,\infty,u_{\epsilon},m).
\end{align*}

\item[$iii$)] For any fixed $\epsilon \in (0,1)$, it holds that
$$
|\mathrm{V}_+(\epsilon,\infty,u_\epsilon,1)-\mathrm{V}_+(0,\infty,u_\epsilon,1)| \leq 4 \epsilon^{2\kappa}\frac{1}{\sigma^2}.
$$

\item[$iv$)] The following limit holds:
\[
\lim_{\epsilon \downarrow 0}\mathrm{V}_+(0,\infty,u_\epsilon,1) =\mathrm{V}_+(0,\infty,\mathbf{D},1).
\]



\end{itemize}
\end{lemma}
\begin{proof}
We start with the proof of \textit{i}).
The statement is a consequence of the following claim:
for any  path $\xi \in A(x_0,T,m)$, there exists a path $\tilde{\xi}$ such that $\tilde{\xi} \in A_+(|x_0|,T,m)$ and $I_{x_0,g}^T({\xi})=I_{|x_0|,g}^T(\tilde{\xi})$.
To verify this, consider $\tilde{\xi}(s)=\xi(s) \cdot \mathbf{I}(\xi(s) \geq 0) -\xi(s) \cdot  \mathbf{I}(\xi(s) \leq 0)$ and observe that $\tilde{\xi} \in A_+(|x_0|,T,m)$.
Then, one easily sees that $I_{|x_0|,g}^T(\tilde{\xi})=I_{x_0,g}^T(\xi).$
Since this  holds for any $\xi \in A(x_0,T,m)$, the proof of our statement follows.

\textit{Proof of ii}).
First, observe that $\mathrm{V}_+(|x_0|,T,u_\epsilon,m) \geq \mathrm{V}_+(|x_0|,\infty,u_\epsilon,m)$ for every $T$: every $\xi \in A(x_0,T,m)$ can be extended to
$\xi \in A(x_0,\infty,m)$ by leaving the value of the variational problem unchanged beyond $T$;
for $s>T$, set $\xi_1(s)=e^{-\frac{1}{\epsilon^{1-\kappa}}(s-T)}\xi(T)$, $\xi_2(s) = (\xi(T)^{1-\kappa}-(s-T)^{\frac{1}{1-\kappa}})\vee 0$. Now, construct
the following path in the case of $\kappa < 1$:
$
\tilde{\xi}(s)=\xi(s) \mathbf{I}(s \in [0,T])+ \mathbf{I}(\{\xi(T) \leq \epsilon\} \cap \{s \geq T\}) \xi_1(s)
+\mathbf{I}(\{\xi(T) > \epsilon\} \cap \{s \geq T\})  \xi_2(s),
$
and take $\tilde{\xi}(s)=\xi(s) \mathbf{I}(s \in [0,T])+ \mathbf{I}(\{s \geq T\})  e^{-(s-T)}\xi(T)$ in the case of $\kappa=1$.
Then, $I_{x, u_{\epsilon}}^{T}({\xi})=I_{x, u_{\epsilon}}^{\infty}(\tilde{\xi})$. This argument also shows that
$\mathrm{V}_+(|x_0|,T,u_\epsilon,m)$ is nonincreasing in $T$.

We next show that for every $\delta>0$, there exists a $T_\delta$ such that $\mathrm{V}_+(|x_0|,T,u_\epsilon,m) \leq \mathrm{V}_+(|x_0|,\infty,u_\epsilon,m) + \delta u_\epsilon (x_0)^2/\sigma^2$ for $T>T_\delta$, which together with
the inequality  $\mathrm{V}_+(|x_0|,T,u_\epsilon,m) \geq \mathrm{V}_+(|x_0|,\infty,u_\epsilon,m)$ implies the desired convergence.
For this, fix $\delta>0$ and pick a path $\xi\in A(x_0,\infty,m)$ such that $I_{x, u_{\epsilon}}^{\infty}({\xi}) < \mathrm{V}_+(|x_0|,\infty,u_\epsilon,m)+\delta$.
Furthermore, let $T'_\delta$ be such that $\int_0^{T'_\delta} |\xi(s)|^p \,\mathrm{d}s \geq m-\delta |x_0|^p$.
Next, define the path $\xi_\delta (s) = x_0 \mathbf{I}(\{s \leq \delta \}) +  \xi(s- \delta)\mathbf{I}(\{s >\delta )$.
The area of this path up to time $T_\delta=T'_\delta + \delta$ is at least $m$, and  $\mathrm{V}_+(|x_0|,T_\delta,u_\epsilon,m) \leq I_{x, u_{\epsilon}}^{T'_\delta}({\xi}) <  \mathrm{V}_+(|x_0|,\infty,u_\epsilon,m)+\delta + \delta u_\epsilon (x_0)^2/\sigma^2$.
The desired convergence now follows by recalling that $\mathrm{V}_+(|x_0|,T,u_\epsilon,m)$ is nonincreasing in $T$.

\textit{Proof of iii}).
We prove $\mathrm{V}_+(0,\infty,u_\epsilon,1)-\mathrm{V}_+(\epsilon,\infty,u_\epsilon,1) \leq  4\epsilon^{2\kappa}\frac{1}{\sigma^2}$; the reverse inequality follows by a similar argument.
Fix $\delta>0$, and pick $\xi_{\epsilon} \in A_+(\epsilon,\infty,1)$ such that $I^{\infty}_{\epsilon, u_{\epsilon}}(\xi_\epsilon)< \mathrm{V}_+(\epsilon,\infty,u_\epsilon,1)+ \delta$.
Set
$
\xi_0(s)=\epsilon \cdot s \cdot  \mathbf{I}(s \in [0,1])
+
\xi_{\epsilon}(s-1) \cdot \mathbf{I}(s > 1)
$
and observe that $\xi_0$ belongs to $A_+(0,\infty, 1)$.
Furthermore, observe that
$
I_{0,u_{\epsilon}}^{\infty}(\xi_0)-I_{\epsilon,u_{\epsilon}}^{\infty}(\xi_{\epsilon}) =
 \frac{1}{\sigma^2}\int_{0}^{1}(\epsilon+ \epsilon^{\kappa}s)^2\mathrm{d}s
  \leq 4 \epsilon^{2\kappa}\frac{1}{\sigma^2}.
$
Therefore,
$
 \mathrm{V}_+(0,\infty,u_\epsilon,1) \leq I_{0,u_{\epsilon}}^{\infty}(\xi_0) \leq I_{\epsilon,u_{\epsilon}}^{\infty}(\xi_\epsilon)+4 \epsilon^{2\kappa}\frac{1}{\sigma^2} \leq \mathrm{V}_+(\epsilon,\infty,u_\epsilon,1)+\delta+ 4\epsilon^{2\kappa}\frac{1}{\sigma^2}.
$
By letting $\delta$ tend to $0$, we obtain the desired inequality.

\textit{Proof of iv}).
Since $u_{\epsilon} \geq \mathbf{D}$ on $A_+(0,\infty,1)$, we have that $I^{\infty}_{0,u_{\epsilon}}(\xi) \leq I^{\infty}_{0,\mathbf{D}}(\xi)$
for any $\xi \in A_+(0,\infty,1)$ and $\epsilon >0$.
Therefore,  $\limsup_{\epsilon \downarrow 0}\mathrm{V}_+(0,\infty,u_\epsilon,1) \leq \mathrm{V}_+(0,\infty,\mathbf{D},1)$.
$\dot{\xi}(s)-u_{\epsilon}(\xi(s))$ is increasing in $\epsilon$ and converging to
$\dot{\xi}(s)-\mathbf{D}(\xi(s))$, as $\epsilon$ tends to 0, point-wise in $s$.
Applying the monotone convergence theorem, we obtain,
for any $\xi \in A_+(0,\infty,1)$,
\begin{align*}
\lim_{\epsilon \downarrow 0}I_{0,u_{\epsilon}}^{\infty}(\xi)
& =
\lim_{\epsilon \downarrow 0}\frac{1}{\sigma^2}\int_{0}^{\infty}|\dot{\xi}(s)-u_{\epsilon}(\xi(s))|^2 \mathrm{d}s
=
 \frac{1}{\sigma^2}\int_{0}^{\infty}|\dot{\xi}(s)-\mathbf{D}(\xi(s))|^2 \mathrm{d}s \geq \mathrm{V}_+(0,\infty,\mathbf{D},1).
\end{align*}
Consequently,
$\liminf_{\epsilon \downarrow 0}\mathrm{V}_+(0,\infty,u_\epsilon,1) \geq \mathrm{V}_+(0,\infty,\mathbf{D},1)$.
\end{proof}

In the same spirit as in Lemma~\ref{continuity-properties-of-V}, we next prove some important properties related to variational problems associated with the various large deviations lower bounds.
To this end,
define, for every subset $I \in [0,\infty)$,
\begin{equation}
     A_{I}(x,T,m)\triangleq\left\{\xi \in \mathbb{C}[0,T]: \int_{0}^{T}\xi(s)^p\mathrm{d}s \geq  m, \  \xi(t) \in I, t\in [0,T],  \ \xi(0)=x\right\}. \label{eq-Ayx}
\end{equation}
\begin{lemma}
\label{continuity-properties-V-2}
\begin{itemize}
\item[$i$)] For both $u=\tilde{u}_{\epsilon}$ and $u={u}_{\epsilon}$, the following holds for every $x \geq \epsilon$,  $x > y$:
	  $$
		\inf_{\xi \in A_{[y,M+x]}(x,T,m)}I_{x,{u}}^{T}(\xi)-\frac{2x^\kappa}{\sigma^2}\left(M+M^\kappa T+x^\kappa T \right)
		 \leq \inf_{\xi \in A_{[0,M]}(0,T,m)}I_{0,\mathbf{D}}^{T}(\xi).
	$$
\item[$ii$)] The following limit holds:
$$
\lim_{T \to \infty}\lim_{M \to \infty}\inf_{\xi \in A_{[0,M]}(x,T,m)}I_{x,\mathbf{D}}^{T}(\xi)=\mathrm{V}(x, \infty,\mathbf{D},m).
$$
\end{itemize}
\end{lemma}
\begin{proof}
For \textit{i}),
let $\xi\in A_{[0,M]}\left(0,T,m\right) $ be such that $\inf_{\xi \in A_{[0,M]}(0,T,m)}I_{0,\mathbf{D}}^{T}(\xi)+\delta> I_{0,\mathbf{D}}^{T}(\xi)$.
Now, let $\tilde{\xi}=\xi+x$. Then, it is obvious that $\tilde{\xi} \in A_{[y,M+x]}(x,T,m)$.
Since $\tilde{\xi}=\xi+x$, $x$ is bigger than $\epsilon$, and $\xi$ is nonnegative, we obtain $\tilde{\xi} \geq \epsilon$ and thus, $u(\tilde{\xi})=\tilde{\xi}^{\kappa}$. This allows us to establish the following:
		\begin{align*}
		I_{x,\tilde{u}_{\epsilon}}^{T}(\tilde{\xi})-I_{0,\mathbf{D}}^{T}(\xi)
		&
		=
		\frac{1}{\sigma^2}
		\int_{0}^{T}\left(\dot{\tilde{\xi}}(s) -\tilde{u}_{\epsilon}(\tilde{\xi}(s)) \right)^2 \mathrm{d}s
		-
		\int_{0}^{T}\left(\dot{{\xi}}(s) -\mathbf{D}({\xi}(s)) \right)^2 \mathrm{d}s
		\\
		&
		=
		\frac{1}{\sigma^2}
		\int_{0}^{T}\left(\dot{\tilde{\xi}}(s) +\tilde{\xi}(s)^{\kappa} \right)^2 \mathrm{d}s
		-
		\int_{0}^{T}\left(\dot{{\xi}}(s) + {\xi}(s)^{\kappa}  \right)^2 \mathrm{d}s
		\\
		&
		\leq
		\frac{1}{\sigma^2}
		\int_{0}^{T}\left(\dot{{\xi}}(s) + {\xi}(s)^{\kappa}+ x^{\kappa} \right)^2 \mathrm{d}s
		-
		\int_{0}^{T}\left(\dot{{\xi}}(s) + {\xi}(s)^{\kappa}  \right)^2 \mathrm{d}s
		\\
		&
		\leq
		\frac{1}{\sigma^2}\int_{0}^{T}
		2x^\kappa
		\left(
		\dot{{\xi}}(s) +{\xi}(s)^{\kappa}
		+ x^\kappa\right) \mathrm{d}s
		\leq
		\frac{2x^\kappa}{\sigma^2}\left(M+M^\kappa T+x^\kappa T \right),
		\end{align*}
where in the first inequality we used the subadditivity of $x \mapsto x^{\kappa}, \ \kappa \in (0,1]$.
Hence, 	$\inf_{\xi \in A_{[0,M]}(0,T,m)}I_{0,\mathbf{D}}^{T}(\xi)+\delta  > 	I_{0,\mathbf{D}}^{T}(\xi) \geq 	I_{x,\tilde{u}_{\epsilon}}^{T}(\tilde{\xi})-\frac{2x^\kappa}{\sigma^2}\left(M+M^\kappa T+x^\kappa T \right)$ which, in turn,  is bigger than or equal to $
\inf_{\xi \in A_{[y,M+x]}(x,T,m)}I_{x,\tilde{u}_\epsilon}^{T}(\xi)$ $-\frac{2x^\kappa}{\sigma^2}\left(M+M^\kappa T+x^\kappa T \right)$.
Finally, since $\delta$ is arbitrary we can let it decrease to $0$, hence, our proof is concluded. The proof for $u_\epsilon$ is similar to the proof
for $\tilde u_\epsilon$ and is therefore omitted.

\textit{Proof of ii}). 
For given $\delta>0$, let $\xi_\delta$ be absolutely continuous such that $I_{x,\mathbf{D}}^{T}(\xi_\delta) < \mathrm{V}_+(x,T,\mathbf{D},m)+\delta$.
Since $\xi_\delta$ is absolutely continuous, its total variation on $[0,T]$ is bounded by some $M_\delta<\infty$.
Since $A_{[0,M]}(x,T,m)$ forms an increasing sequence w.r.t. $M$, we see that $\inf_{\xi \in A_{[0,M]}(x,T,m)}I_{x,\mathbf{D}}^{T}(\xi)< \mathrm{V}_+(x,T,\mathbf{D},m)+\delta$
for $M\geq M_\delta$, establishing the limit w.r.t.\ $M$.
The limit when $T\rightarrow\infty$ follows by the same argument as in the proof of Lemma~\ref{continuity-properties-of-V}, item $ii$).
\end{proof}

\section{Proof of Lemma~\ref{semi-exponential-property-cycles}}\label{proof-lemma-21}
In this section, we prove the asymptotic semi-exponential property of the single, full cycle $C_1^{\delta}$.
Specifically, in Section~\ref{upper-bound-lemma2.1}
we develop the large deviations upper bound for $C_1^\delta$;
in Section~\ref{lower-bound-lemma2.1}, we establish the large deviations lower bound for $C_1^\delta$;
and in Section~\ref{proof-for-ld-estimates-c1d}, we conclude the proof of Lemma~\ref{semi-exponential-property-cycles}.

\subsection{Upper Bound for the Full Cycle $C_1^\delta$}\label{upper-bound-lemma2.1}

In this subsection, we prove the large deviations upper bound for the full cycle $C_1^{\delta}$.
Our approach consists of using sample-path analysis in combination with LDP Result~\ref{FW-SPLDP}.
However, since $X$ has non-Lipschitz drift, and the diffusion term of $X$ does not decrease to $0$, $X$ does not satisfy the framework of Result~\ref{FW-SPLDP}.
To overcome this, we use an appropriate scaling, yielding $X_t$, such that the diffusion term tends to $0$ (see Lemma~\ref{distributional-equalities-via-scaled-diffusion-process}); and we introduce a modified diffusion process $X_{t,\epsilon}$ with a Lipschitz drift whose area serves as a stochastic upper bound to $C_1^\delta$ (see Lemma~\ref{stochastic-upper-bound-via-modified-dif-n-iid-cjs}).

\subsubsection{Useful distributional equalities and an upper bound}\label{scaling-n-crude-bounds}

In the following lemma
we state a useful distributional equality between $X$ and $X_{t}$ (defined in \eqref{eq-representation-Xt}) and provide a first upper bound for $C_1^\delta$.
For notational convenience, define the conditional probability $\P_{y}(Y \in \cdot)= \P(Y \in \cdot \big| Y(0)=y)$.
\begin{lemma}\label{distributional-equalities-via-scaled-diffusion-process}
For every fixed $p>2\kappa$, $\delta >0$, let
 $T_0\left(t,\frac{\delta}{t^{\alpha/p}}\right)= \inf\left\{s \geq 0: |X_t(s)|=\frac{\delta}{t^{\alpha/p}}\right\}$, and define the hitting time
	 $\tau_t(0) = \inf\left\{s \geq T_0\left(t,\frac{\delta}{t^{\alpha/p}}\right): |X_t(s)|=0\right\}$.
	 Then:
	\begin{itemize}
		\item[$i$)] It holds that
		 \begin{equation*}
		\P_0\left(\int_{0}^{B_1^{\delta}}|X(s)|^{p} \mathrm{d}s \geq  t\right)
		=
		\P_0\left(\int_{0}^{\tau_t(0)}|X_{t}(s)|^{p} \mathrm{d}s \geq  1\right).
		\end{equation*}
		\item[$ii$)] For every fixed $\epsilon>0$ and sufficiently large $t$,
		
		\begin{equation*}
		\P_0\left(\int_{0}^{T_0\left(t,\frac{\delta}{t^{\alpha/p}}\right)}|X_{t}(s)|^{p} \mathrm{d}s > \epsilon \right) \leq
		\bigO\left(\exp\left( -t\epsilon{\frac{\delta^{2\kappa-p}}{\sigma^2}}\right)\right).
		\end{equation*}
	\end{itemize}
\end{lemma}

\begin{proof}
The proof of  \textit{i}) 
follows directly by applying the definitions and a change of variables argument. \\
\textit{Proof of ii}). From $T_0\left(t,\frac{\delta}{t^{\alpha/p}}\right)= \inf\left\{s \geq 0: |X_t(s)|=\frac{\delta}{t^{\alpha/p}}\right\}$ and $X_t(0)=0$, it follows that over the interval $\Big[0,T_0\left(t,\frac{\delta}{t^{\alpha/p}}\right)\Big]$, the stochastic process $X_t$ is bounded by $\frac{\delta}{t^{\alpha/p}}$. Consequently, 
%
	\begin{align*}
	&
	\P_0\left(\int_{0}^{T_0\left(t,\frac{\delta}{t^{\alpha/p}}\right)}|X_{t}(s)|^{p} \mathrm{d}s > \epsilon \right)
	 \leq \P_0\left( \frac{\delta^p}{t^{\alpha}} T_0\left(t,\frac{\delta}{t^{\alpha/p}}\right) > \epsilon \right)
	\leq
	\P_0\left(|X_t(s)| \leq \frac{\delta}{t^{\alpha/p}} \ \text{ for } s \in \ \left[0,\frac{\epsilon}{\delta^p}t^{\alpha}\right]\right).
	\end{align*}
To deal with the event in the last inequality, we derive an upper bound using a Brownian motion with drift.
To this end, recall that $X_{t}$ satisfies the following integral equation: $
X_t(u)=-\int_{0}^{u}sgn\left(X_t(s)\right)\left|X_t(s)\right|^{\kappa}\mathrm{d}s
	+
	t^{-\gamma/2}2^{-1/2} \sigma B(u)$, $u\geq 0$. 
In view of Theorem~1.1 in \cite{yamada1973comparison}, if $|X_{t}(s)|$ is less than $\frac{\delta}{t^{\alpha/p}}$ over $s \in \ \left[0,\frac{\epsilon}{\delta^p}t^{\alpha}\right]$, we obtain that
	\begin{align*}
&\P_0\left(|X_t(s)| \leq \delta/t^{\alpha/p},  s\in \left[0,\frac{\epsilon}{\delta^{p}}t^{\alpha}\right]\right)
\leq
 	\P_0\left(|t^{-\gamma/2} 2^{-1/2}\sigma B(s)-\left(\delta/t^{\alpha/p}\right)^{\kappa}s| \leq \frac{\delta}{t^{\alpha/p}} \ \text{for} \ s \in \left[0,\frac{\epsilon}{\delta^p}t^{\alpha}\right] \right)
\\
	&
	\leq
	\P_0\left(\sup_{s \leq \frac{\epsilon}{\delta^p}t}
	\left\{|
	2^{-1/2}\sigma B(s)-\delta^\kappa s |\right\} \leq \delta
	\right)
	\leq \exp\left(-t \delta^{2\kappa-p}\epsilon/\sigma^2 + 2\delta^{\kappa+1}/\sigma^2\right) ,
	\end{align*}
where in the second inequality we have used the self-similarity of Brownian motion and the relationships between $\alpha,\beta,\gamma$, and in the last inequality we have used a standard change of measure argument with respect to a zero drift Brownian motion.
This yields the statement of the lemma.
\end{proof}

\subsubsection{An upper bound using the modified diffusion process $X_{t,\epsilon}$}\label{upper-bounds-with-Xte}
In this subsection, we construct a diffusion process $X_{t,\epsilon}$ such that it has a Lipschitz continuous drift function, and we use its area as a stochastic upper bound for $C_1^\delta$.
We first introduce some necessary notation.
Let $X_{t,\epsilon}$ be the solution of the following SDE:
\begin{equation}
\label{eq-def-Xte}
\mathrm{d}X_{t,\epsilon}(s)=u_{\epsilon}\left(X_{t,\epsilon}(s)\right)\mathrm{d}s+t^{-\gamma/2} 2^{-1/2}\sigma \mathrm{d}B(s), \  \epsilon>0, \ s \geq 0,
\end{equation}
where the drift $u_\epsilon$ has been defined in \eqref{eq-modifieddrift}.
For the diffusion process $ X_{t,\epsilon}$, denote with $\tau_{t,\epsilon}(\delta)=\inf\{s \geq 0: X_{t,\epsilon}(s)=\delta\}$ i.e., its first hitting time to $\delta$.
Let  $N_g(p_{t,\epsilon})$ denote a geometric random variable with success probability $\P_{\epsilon/2}\left(\tau_{t,\epsilon}(\epsilon) > \tau_{t,\epsilon}(0) \right)$.
Finally, let  $C_{j,t}, j\geq 1$, be i.i.d.~copies of
\begin{equation}
\label{eq-def-ct}
    C_t=\int_{0}^{\tau_{t,\epsilon}(\epsilon/2)}\big(X_{t,\epsilon}(s)\big)^p\mathrm{d}s \mbox{ subject to } X_{t,\epsilon}(0)=\epsilon.
\end{equation}
\begin{lemma}\label{stochastic-upper-bound-via-modified-dif-n-iid-cjs} Let $b \geq 0$, $\delta>0$ and $\epsilon>0$ be fixed.
Then,
for every fixed $\epsilon_0>0$ and sufficiently large $t$ $(t \geq (\delta/\epsilon)^{p/\alpha})$,
	\begin{align*}
	&
	\P_{\frac{\delta}{t^{\alpha/p}}}
	\left(
	\int_{0}^{\tau_{t}(0)}\big(X_{t}(s)\big)^p  \mathrm{d}s
	\geq b \right) \nonumber
	\leq
	\underbrace{\P_{\frac{\delta}{t^{\alpha/p}}}
		\left(
		\int_{0}^{\tau_{t}(0)}\big(X_{t}(s)\big)^p  \mathbf{I}\left(X_{t}(s) \leq \epsilon \right) \mathrm{d}s \geq \epsilon_0
		\right)}_{\triangleq T(I)}
 +  \underbrace{\P_\epsilon \left(
		\sum_{j=1}^{N_{g}(p_{t,\epsilon})} C_{j,t} \geq b-\epsilon_0
		\right)}_{\triangleq T(II)}.
	\end{align*}

\end{lemma}
\begin{proof}
We first observe that the following inequality holds:
	\begin{align}\label{stochastic-dominance-1-inequality}
	&\P_{\frac{\delta}{t^{\alpha/p}}}\left(\int_{0}^{\tau_t(0)}\big(X_t(s)\big)^{p}\cdot\mathbf{I}\left(X_t(s)\geq \epsilon \right)\mathrm{d}s \geq b \right)
	\leq
	\P_{\frac{\delta}{t^{\alpha/p}}}\left(\int_{0}^{\tau_{t,\epsilon}(0)}\big(X_{t,\epsilon}(s)\big)^{p}\cdot\mathbf{I}\left(X_{t,\epsilon}(s)\geq \epsilon \right)\mathrm{d}s \geq  b \right).
	\end{align}
To see this, observe that
the drift function $u_\epsilon$ is bigger than or equal to $\mathbf{D}$ in the relevant region.
Thus, using Theorem~1.1 in \cite{yamada1973comparison}, $X_{t,\epsilon}$ is stochastically bigger than $X_{t}$.
Consequently, the first hitting time to zero, $\tau_{t,\epsilon}(0)$, is stochastically bigger than $\tau_{t}(0)$ and, hence, (\ref{stochastic-dominance-1-inequality}) follows.

The proof of the lemma now follows from (\ref{stochastic-dominance-1-inequality})  once we have proven that, for sufficiently large $t$ $(t \geq (\delta/\epsilon)^{p/\alpha})$,
	\begin{equation}\label{stochastic-dominance-2-inequality}
	\P_{\frac{\delta}{t^{\alpha/p}}}
	\left(
	\int_{0}^{\tau_{t,\epsilon}(0)}\big(X_{t,\epsilon}(s)\big)^p \cdot \mathbf{I}\left(X_{t,\epsilon}(s) \geq \epsilon \right) \mathrm{d}s \geq b
	\right)
	\leq
	\P_\epsilon\left(
	\sum_{j=1}^{N_{g}(p_{t,\epsilon})} C_{j,t} \geq b
	\right).
	\end{equation}
First, choose $t$ large enough so that $\frac{\delta}{t^{\alpha/p}}$ is smaller than $\epsilon$, and consider the diffusion process $X_{t,\epsilon}$ initialized at $\epsilon$.
Since every trajectory of $X_{t,\epsilon}$ initialized at $\frac{\delta}{t^{\alpha/p}}$ can be vertically transposed to form a trajectory of $X_{t,\epsilon}$ initialized at $\epsilon$, we can deduce the following distributional inequality:
\begin{align*}
	&
	\P_{\frac{\delta}{t^{\alpha/p}}}
	\left(
	\int_{0}^{\tau_{t, \epsilon}(0)}\left(X_{t,\epsilon}(s)\right)^p \cdot\mathbf{I}\left(X_{t,\epsilon}(s) \geq \epsilon \right) \mathrm{d}s
	\geq b \right)
	\leq
	\P_{\epsilon}
		\left(
		\int_{0}^{\tau_{t, \epsilon}(0)}\left(X_{t,\epsilon}(s)\right)^p \cdot \mathbf{I}\left( X_{t,\epsilon}(s) \geq \epsilon \right) \mathrm{d}s \geq b
		\right).
	\end{align*}	
	Now, we provide an upper bound for the RHS of the above inequality.
	Towards this, let us introduce the following stopping times:
	$A_1 \triangleq 0$, 
	 $B_1 \triangleq \inf\{s \geq A_1: X_{t,\epsilon}(s)= \epsilon/2\}$, 
	 $A_2 \triangleq \inf\{s \geq B_1: X_{t,\epsilon}(s)=\epsilon\}$, 
	$B_2 \triangleq \inf\{s \geq A_2: X_{t,\epsilon}(s)= \epsilon/2\}$, 
	 $A_i \triangleq \inf\{s \geq B_{i-1}: X_{t,\epsilon}(s)=\epsilon\}$, 
	 $B_i \triangleq \inf\{s \geq A_i: X_{t,\epsilon}(s)= \epsilon/2\}$. 
Based on the stopping times $\{A_i\}_{i \geq 1}$, $\{B_i\}_{i \geq 1}$, we obtain the following decomposition:
	$
	\int_{0}^{\tau_{t,\epsilon}(0)}\left(X_{t,\epsilon}(s)\right)^p \cdot \mathbf{I}\left(X_{t,\epsilon}(s) \geq \epsilon \right) \mathrm{d}s
	\leq \sum_{i=1}^{N(p_t)-1} \int_{A_i}^{B_i}\left(X_{t,\epsilon}(s)\right)^p \mathrm{d}s
	$
where $N(p_t)$ is the number of cycles, or intervals of the form $[A_i,B_i]$, generated during the time horizon $[0, \tau_{t,\epsilon}(0)]$.
That is,
	$
	N(p_t)=\inf\{ i \geq 1 :A_{i} \geq \tau_{t,\epsilon}(0)\}.
	$
Due to the strong Markov property of $X_{t,\epsilon}$,
$\P\left(\int_{A_i}^{B_i}\left(X_{t,\epsilon}(s)\right)^p \mathrm{d}s\in \Gamma\right)=\P_\epsilon \left(C_{1,t} \in \Gamma\right)$, and
$N(p_t)$ is stochastically dominated by the geometric random variable $N_g(p_{t,\epsilon})$, which has success probability equal to
$\P_{\epsilon/2}\left(\tau_{t, \epsilon}(\epsilon) > \tau_{t,\epsilon}(0)\right)$.
Combining these observations proves  \eqref{stochastic-dominance-2-inequality}.
\end{proof}

\subsubsection{The large deviations upper bound for the area swept under $X_{t,\epsilon}$}\label{LD-for-area-of-Xte}

Due to Lemma~\ref{stochastic-upper-bound-via-modified-dif-n-iid-cjs}, we can now derive a convenient upper bound for $C_1^\delta$.
Towards this purpose, we obtain the log-asymptotics for each of the terms $T(I)$ and $T(II)$ appearing in 
Lemma~\ref{stochastic-upper-bound-via-modified-dif-n-iid-cjs}.
The following lemma provides an asymptotic upper bound for the area of the diffusion process $X_t$ when considering its trajectories that are below $\epsilon$.

\begin{lemma}\label{area-for-epsilon-bounded-trajectories}
For every fixed $p >2\kappa$, $\delta>0$ 
and $\epsilon_1 >0$,   
	it holds that
	\begin{align*}
	& \lim_{\epsilon \downarrow 0}	\limsup_{t \to \infty}  t^{-\gamma} \log\P_{\frac{\delta}{t^{\alpha/p}}}\left(\int_{0}^{\tau_{t}(0)}
	\big(X_{t}(s)\big)^{p} \cdot \mathbf{I}\left( X_t(s)  \leq \epsilon\right) \mathrm{d}s > \epsilon_1 \nonumber
	\right)
=-\infty.
	\end{align*}
\end{lemma}

\begin{proof}
Define $\zeta\triangleq 2\kappa /(1+(1-\kappa)/p)$. Then $0<\zeta/p=\gamma + \alpha - 1 < 1$ and we can write
\begin{align*}
\P_{\frac{\delta}{t^{\alpha/p}}}\left(\int_{0}^{\tau_{t}(0)}
	\big(X_{t}(s)\big)^{p} \cdot \mathbf{I}\left( X_t(s)  \leq \epsilon\right) \mathrm{d}s > \epsilon_1
	\right)  &\leq \P_{\frac{\delta}{t^{\alpha/p}}}\left(\int_{0}^{\tau_{t}(0)}
	\big(X_{t}(s)\big)^{\zeta}  \mathrm{d}s > \epsilon_1 \epsilon ^{\zeta - p}
	\right) \\
 &= \P_{\delta}\left(\int_{0}^{B_1^\delta}
	\big(X(s)\big)^{\zeta}  \mathrm{d}s > t^{\gamma} \epsilon_1 \epsilon ^{\zeta - p}
	\right).
\end{align*}
We can upper-bound the final expression further by
\begin{equation}
\P_{\delta}\left(\int_{0}^{\min \{B_1^\delta, t^\gamma \epsilon^{(\zeta-p)/2}\}}
	\big(X(s)\big)^{\zeta}  \mathrm{d}s > t^{\gamma} \epsilon_1 \epsilon ^{\zeta - p}
	\right) + \P_{\delta}(B_1^\delta > t^\gamma \epsilon^{(\zeta-p)/2} ).
\end{equation}
Due to the finiteness of the m.g.f.\ of $B_1^\delta$ (see Lemma~\ref{bounded-moment-generating-function-t-1-delta}), and since $\zeta-p<0$, we have that
\begin{equation}
    \lim_{\epsilon\downarrow 0} \lim_{t\rightarrow\infty} t^{-\gamma} \log  \P_{\delta}(B_1^\delta > t^\gamma \epsilon^{(\zeta-p)/2} )=-\infty.
\end{equation}
Set $T=t^\gamma \epsilon^{(\zeta-p)/2}$.
We proceed by noting that
\begin{equation}
    \P_{\delta}\left( \frac 1T \int_{0}^{\min \{B_1^\delta, T\}}
	\big(X(s)\big)^{\zeta}  \mathrm{d}s >  \epsilon_1 \epsilon ^{(\zeta - p)/2}
	\right) \leq \P^* \left( \frac 1T \int_{0}^{ T}
	\big(|X(s)|\big)^{\zeta}  \mathrm{d}s >  \epsilon_1 \epsilon ^{(\zeta - p)/2}  \right) /\P^*(X(0) > \delta),
\end{equation}
with $\P^*$  the measure under which $X$ is a stationary diffusion.
In what follows we distinguish between $\kappa=1$ and $\kappa<1$.\\

{\em The case $\kappa=1$.} Our SDE reduces to a stationary OU process, and since $\zeta=2$, this falls within the framework of large deviations of quadratic functionals of Gaussian processes, which are investigated in \cite{bryc1997large}.
In particular, Theorem~1 in \cite{bryc1997large} and the OU example below it, imply that
$\frac 1T \int_{0}^{T}
	\big(X(s)\big)^{2}  \mathrm{d}s$ satisfies an LDP with rate function $(\sigma \sqrt{x}-1/(\sigma \sqrt{x}))^2/8$, see also Result~\ref{dembobryc} below.
 Consequently, if $\kappa=1$, for every fixed $\epsilon>0$,
 \begin{equation}
     \frac 1T \log  \P^{\ast}\left( \frac 1T \int_{0}^{T}
	\big(|X(s)|\big)^{\zeta}  \mathrm{d}s >  \epsilon_1 \epsilon ^{(\zeta - p)/2}
	\right) \rightarrow - \frac 18 (\sigma \sqrt{\epsilon_1 \epsilon ^{(\zeta - p)/2}}-1/(\sigma \sqrt{\epsilon_1 \epsilon ^{(\zeta - p)/2}}))^2,
 \end{equation}
as $T\rightarrow\infty$, and the expression on the RHS converges to $-\infty$ as $\epsilon \downarrow 0$, since $\zeta < p$.\\

{\em The case $\kappa<1$.} In this case, we rely on the more recent work in \cite{wuyiao2008}, where light-tailed large deviations for unbounded additive functionals are derived using $\Phi$-Sobolev inequalities.
Specifically, our diffusion for $\kappa<1$ falls within the framework of Theorem~2.1 (see also Result~\ref{wuyiao} below), and in particular Example~3.2 in \cite{wuyiao2008}, where it is shown that our diffusion satisfies~(2.7) of \cite{wuyiao2008} (displayed as condition~\eqref{condition-wu1} below).
Since $X(\infty)$ has a density proportional to $\exp\{ - c |x|^{\kappa+1}\}$ and $\zeta<\kappa +1$, we have that $\E\left(\exp \{\lambda |X(\infty)|^\zeta\}\right)<\infty$ for all $\lambda>0$, and from (3.1) in \cite{wuyiao2008}, it follows that~(2.8) of \cite{wuyiao2008} (displayed as condition~\eqref{condition-wu2} below) is satisfied.
Because we work under the stationary measure $\P^*$,
%
we can now apply Theorem~2.1 of \cite{wuyiao2008} (see also Result~\ref{wuyiao}) to conclude that there exists a good rate function $J$ such that
\begin{equation}
\limsup_{T\rightarrow\infty} \frac 1T \log      \P^{\ast}\left( \frac 1T \int_{0}^{T}
	\big(|X(s)|\big)^{\zeta}  \mathrm{d}s \geq  \epsilon_1 \epsilon ^{(\zeta - p)/2}	\right) \leq - J\left( (\epsilon_1 \epsilon ^{(\zeta - p)/2} - \E(|X(\infty)|^\zeta))^+\right).
\end{equation}
Since $J$ has compact level sets (Theorem~2.1, item $i$), of \cite{wuyiao2008}), the RHS converges to $-\infty$ as $\epsilon \downarrow 0$.

Putting everything together, we can conclude the assertion.
\end{proof}

Now, we turn our focus to the tail asymptotics of $C_{t}$ defined in \eqref{eq-def-ct}, 
cf.~$T(II)$ in 
Lemma~\ref{stochastic-upper-bound-via-modified-dif-n-iid-cjs}.
As a preparation, the next lemma provides an asymptotic estimate for the tail probability of $\tau_{t,\epsilon/2}(\epsilon/2)$.

\begin{lemma}\label{high-occupation-interval-asymptotics}
	Assume that $X_{t,\epsilon/2}(0)=\epsilon$.
	Then, for any fixed $H>0$ and $\epsilon>0$ such that $H \geq {4}{\epsilon^{1-\kappa}}$, it holds that
	\begin{equation*}
	\limsup_{t \to \infty} t^{-\gamma}
	\log\P_{\epsilon}\left(
	\tau_{t,\epsilon/2}(\epsilon/2) \geq  H
	\right)
	\leq
	-\frac{1}{\sigma^2} \left(
	\frac{\epsilon^{2 \kappa}}{4} \cdot H
	-
	\epsilon^{1+\kappa}\right).
	\end{equation*}
	\end{lemma}
\begin{proof}
	Recall that the modified diffusion process $X_{t,\epsilon/2}(\cdot)$ satisfies \eqref{eq-def-Xte}.
	The modified diffusion satisfies the conditions of the LDP in Result~\ref{FW-SPLDP}; the LDP holds in the space of continuous functions $\mathbb{C}[0,T]$ equipped with the supremum norm, with speed $t^\gamma$, and rate function $I_{\epsilon, u_{\epsilon/2}}^T$.
For the estimation of the hitting time asymptotics, we note that
	$
	\left\{ \tau_{t,\epsilon/2}(\epsilon/2) \geq  H \right\}
	=
	\left\{X_{t,\epsilon/2}(s) > \epsilon/2, s\in [0,H] \right\}.
	$
	%
	%
	This enables us to use the large deviations framework.
	To do so, we define the set $S_H \triangleq \{\xi \in \mathbb{C}[0,H]: \xi(0)=\epsilon, \ \xi(s) \geq \epsilon/2 \ \text{for every} \ s \in [0,H]\}$, and observe that $
	\P_{\epsilon}
	\left(
	X_{t,\epsilon/2}(s) > \epsilon/2  \ \text{for} \ s \in [0,H]\right)
$ is bounded by $
	\P_{\epsilon}\left( X_{t,\epsilon/2}
	\in S_H
	\right).
	$
	Therefore, due to the large deviations principle for $X_{t,\epsilon/2}$,
	\begin{equation*}
	\limsup_{t \to \infty}  t^{-\gamma} \log\P_{\epsilon}
	\left(
	X_{t , \epsilon/2}  \in S_H
	\right)
	\leq
	-\inf_{\xi \in S_H}I_{\epsilon, u_{\epsilon/2}}^H(\xi).
	\end{equation*}
It remains to derive a lower bound for
		$
		\inf_{\xi \in S_H}I_{\epsilon, u_{\epsilon/2}}^H(\xi).
		$
	To this end, for any path $\xi \in S_H$ such that $I_{\epsilon, u_{\epsilon/2}}^{H}(\xi)< \infty$, we have that $u_{\epsilon/2}(\xi(s))=-|\xi(s)|^\kappa, \ s \geq 0
$. That is,
	\begin{align*}
	I_{\epsilon, u_{\epsilon/2}}^H(\xi)
	&
	=
	\int_{0}^{H}\frac{|\dot{\xi}(s)-u_{\epsilon/2}(\xi(s))|^2 }{\sigma^2}\mathrm{d}s
	=
	\int_{0}^{H}\frac{(\dot{\xi}(s)+|u_{\epsilon/2}(\xi(s))|)^2 }{\sigma^2}\mathrm{d}s \nonumber
	\\
	&	
	=
	\frac{1}{\sigma^2} \left(
	\int_{0}^{H}\left(u_{\epsilon/2}(\xi(s))\right)^2\mathrm{d}s
	+
	2\int_{0}^{H}\dot{\xi}(s)|u_{\epsilon/2}(\xi(s))|\mathrm{d}s
	+
	\int_{0}^{H}(\dot{\xi}(s))^2\mathrm{d}s \right)
	\nonumber
	\\
	&
	\geq
	\frac{1}{\sigma^2} \left(
	\int_{0}^{H}\min\left\{\frac{\xi^2(s)}{(\epsilon/2)^{2(1-\kappa)}},\xi^{2 \kappa}(s)\right\}\mathrm{d}s
	+
	2\int_{0}^{H}\dot{\xi}(s)|u_{\epsilon/2}(\xi(s))|\mathrm{d}s
	+
	\int_{0}^{H}(\dot{\xi}(s))^2\mathrm{d}s \right)
	\nonumber
	\\
	&
	\geq
	\frac{1}{\sigma^2} \left(
	\int_{0}^{H}\min\left\{\frac{\xi^2(s)}{(\epsilon/2)^{2(1-\kappa)}},\xi^{2 \kappa}(s)\right\}\mathrm{d}s
	+
	2\int_{0}^{H}\dot{\xi}(s)\min\left\{\frac{\xi(s)}{(\epsilon/2)^{1-\kappa}},\xi^{ \kappa}(s)\right\}\mathrm{d}s\right)
	\nonumber
	\\
	&
	\geq
	\frac{1}{\sigma^2} \left(
	\int_{0}^{H}\frac{\epsilon^{2\kappa}}{4}\mathrm{d}s
	+
	\epsilon^{\kappa}\left(\xi(H)-\xi(0)\right)\right) \nonumber
	\geq
	\frac{1}{\sigma^2} \left(
	\frac{\epsilon^{2\kappa}}{4} \cdot H
	-
	\epsilon^{1+\kappa}\right). \nonumber
	\end{align*}
	Since the lower bound holds for every $\xi \in S_H$, the assertion follows.
	\end{proof}

\begin{lemma}\label{tail-asymptotics-for-semicycle}
For every fixed $p>2\kappa$ and $b \geq 0$, it holds that
	\begin{equation*}
	\lim_{\epsilon \downarrow 0}\lim_{t \to \infty}\frac{1}{t^{\gamma}}\log\P_{\epsilon}\left(
	\int_{0}^{\tau_{t,\epsilon/2}(\epsilon/2)}|X_{t,\epsilon/2}(s)|^p \mathrm{d}s \geq b
	\right) = -b^{\gamma}\cdot \mathrm{V}(0,\infty,\mathbf{D}, 1).
	\end{equation*}
	\end{lemma}
\begin{proof}
We focus on the case $b=1$.
The case $b \neq 1$ can be reduced to the case $b=1$ by considering the rescaled process $\frac{X_{t,\epsilon/2}(\cdot\  b^{\beta})}{b^{\alpha/p}}$ over $[0,\tau_{t,\epsilon/2}(\epsilon/2)]$.
To obtain an upper bound, we write
	\begin{align}\label{crude-upper-bound-S}
	&
	\P_{\epsilon}\left(\int_{0}^{\tau_{t,\epsilon/2}(\epsilon/2)}|X_{t,\epsilon/2}(s)|^p \mathrm{d}s \geq 1 \right) \nonumber
	\\
	& \leq
	\underbrace{\P_{\epsilon}\left( \int_{0}^{\tau_{t,\epsilon/2}(\epsilon/2)}|X_{t,\epsilon/2}(s)|^p \mathrm{d}s \geq 1, \tau_{t,\epsilon/2}(\epsilon/2) \leq H\right)}_{\triangleq (i)}
 +
	\underbrace{\P_{\epsilon}\left(\tau_{t,\epsilon/2}(\epsilon/2) \geq H \right)}_{\triangleq (ii)}.
	\end{align}
	Next, we obtain large deviations estimates for terms $(i)$ and $(ii)$. 	
	For $(ii)$, due to Lemma~\ref{high-occupation-interval-asymptotics},
	\begin{equation}\label{asymptotics-for-term-II}
	\limsup_{t \to \infty}t^{-\gamma} \log\P_{\epsilon}\left(\tau_{t,\epsilon/2}(\epsilon/2) \geq H \right) \leq -\frac{1}{\sigma^2} \left(
	\frac{\epsilon^{2\kappa}}{4} \cdot H
	-
	\epsilon^{1+\kappa}\right).
	\end{equation}
	To deal with $(i)$, we can easily provide an upper bound for the area functional:
	\begin{align}\label{asymptotics-for-term-1}
&	\P_{\epsilon}\left( \int_{0}^{\tau_{t,\epsilon/2}(\epsilon/2)}|X_{t,\epsilon/2}(s)|^p \mathrm{d}s \geq 1, \tau_{t,\epsilon/2}(\epsilon/2) \leq H \right)  
\leq
	\P_{\epsilon}\left(\int_{0}^{H}|X_{t,\epsilon/2}(s)|^p \mathrm{d}s \geq 1  \right).
	\end{align}
	In view of the sample-path LDP for $X_{t,\epsilon/2}$ (Result~\ref{FW-SPLDP}), we apply the contraction principle with the continuous functional $\int_{0}^{H}|\xi(s)|^p\mathrm{d}s$, and we obtain
	\begin{equation}\label{large-deviation-estimate-term1}
	\limsup_{t \to \infty}t^{-\gamma}\log\P_{\epsilon}\left(
	\int_{0}^{H}|X_{t,\epsilon/2}(s)|^p \mathrm{d}s \geq 1
	\right) \leq - \inf_{\xi \in A\big(\epsilon,H,1\big)}I_{\epsilon, u_{\epsilon/2}}^{H}(\xi),
	\end{equation}
recalling \eqref{eq-regionvariationalproblem}.
Combining \eqref{crude-upper-bound-S}, \eqref{asymptotics-for-term-II}, \eqref{asymptotics-for-term-1}, \eqref{large-deviation-estimate-term1}, and the principle of the largest term, we obtain
	\begin{align*}
	&
	\limsup_{t \to \infty}t^{-\gamma}\log\P_{\epsilon}\left(
	\int_{0}^{\tau_{t,\epsilon/2}(\epsilon/2)}|X_{t,\epsilon/2}(s)|^p \mathrm{d}s \geq 1
	\right) \nonumber
	\\
	&
	\leq
	\max\left\{-\lim_{H \to \infty}\inf_{\xi \in A\big(\epsilon,H,1\big)}I_{\epsilon, u_{\epsilon/2}}^{H}(\xi), \ -\lim_{H \to \infty}\frac{1}{\sigma^2} \left(
	\frac{\epsilon^{2\kappa}}{4} \cdot H
	-
	\epsilon^{1+\kappa}\right)\right\} \nonumber
	\\
	&
	=
	-\lim_{H \to \infty}\inf_{\xi \in A\big(\epsilon,H,1\big)}I_{\epsilon,u_{\epsilon/2}}^{H}(\xi)=-\lim_{H \to \infty}\mathrm{V}\big(\epsilon,H,u_{\epsilon/2}, 1\big) \rightarrow -\mathrm{V}(0,\infty,\mathbf{D}, 1),
	\end{align*}
as $\epsilon \downarrow 0$, 
 invoking all the properties of Lemma~\ref{continuity-properties-of-V}.
	%
%
%
	
 For the lower bound, let $L_{H,\epsilon/2}=\{\inf_{s \in [0,H]}\{{X}_{t,\epsilon/2}(s)\} > \epsilon/2\}$, and define $U_{H,\epsilon}=\{\sup_{s \in [0,H]}\{{X}_{t,\epsilon/2}(s)\} < M+\epsilon\}$. Then,
  \begin{align*}\label{crude-upper-bound-2-S}
\liminf_{t \to \infty}t^{-\gamma}\log\P_{\epsilon}\left(\int_{0}^{\tau_{t,\epsilon/2}(\epsilon/2)}|X_{t,\epsilon/2}(s)|^p \mathrm{d}s > 1 \right) \nonumber
	&
	 \geq
	\liminf_{t \to \infty}t^{-\gamma}\log\P_{\epsilon}\left( \int_{0}^{H}|X_{t,\epsilon/2}(s)|^p \mathrm{d}s > 1, \ L_{H,\epsilon/2}, U_{H,\epsilon} \right)  \nonumber
	\\
\geq
-\inf_{\xi \in A_{[\epsilon/2, M+\epsilon]}\big(\epsilon,H,1 \big)}	I_{\epsilon, u_{\epsilon/2}}^H(\xi) &\geq -\inf_{\xi \in A_{[0,M]}(0,H,1)}I_{0,\mathbf{D}}^{H}(\xi)-\frac{2\epsilon^\kappa}{\sigma^2}\left(M+M^\kappa H+{\epsilon^{\kappa}} H \right),
	\end{align*}
	recalling definition~\eqref{eq-Ayx}, and applying
 \textit{i}) of Lemma~\ref{continuity-properties-V-2};  we set $y=\epsilon/2$, $x=\epsilon$ in the last step.
 The assertion now follows by, respectively, letting  $\epsilon\downarrow 0$,
$M\rightarrow \infty$, and $H\rightarrow\infty$, and invoking \textit{ii}) of Lemma~\ref{continuity-properties-V-2}.
\end{proof}

The next lemma constitutes the last main component for the log-asymptotics related to the upper bound for the full cycle $C_{1}^{\delta}$.
Let now, for $j\geq 1$, $C_{j,t}$'s be i.i.d.\ copies of $C_{t}=\int_{0}^{\tau_{t,\epsilon/2}(\epsilon/2)}|X_{t,\epsilon/2}(s)|^p \mathrm{d}s$, subject to $X_{t,\epsilon/2}(0)=\epsilon$, and recall that $N_{g}(p_{t,\epsilon/2})$ is a geometric random variable with success probability $p_{t,\epsilon/2}=\P_{\epsilon/2}(\tau_{t,\epsilon/2}(\epsilon) > \tau_{t,\epsilon/2}(0))$.
The following lemma establishes tail asymptotics for the sum $\sum_{j=1}^{N_{g}(p_{t,\epsilon/2})} C_{j,t}$.

\begin{lemma}\label{log-asy-sum-of-eps-semi-cycles}
For every fixed $p>2\kappa$ and $b \geq 0$, it holds that
\[
\lim_{\epsilon \downarrow 0}\lim_{t \to \infty}t^{-\gamma}\log\P\left(\sum_{j=1}^{N_{g}(p_{t,\epsilon/2})} C_{j,t} \geq b\right)=-b^{\gamma} \cdot \mathrm{V}(0,\infty,\mathbf{D}, 1).
\]
\end{lemma}
Note that this lemma does not follow from standard large deviations results for sums of semi-exponential random variables: the $C_{j,t}$'s depend on $\epsilon$ through the modified diffusion process $X_{t,\epsilon}$.

\begin{proof}
%
In Lemma~\ref{tail-asymptotics-for-semicycle}, we have verified the semi-exponential property for each of the $C_{j,t}$'s.
For the proof of our statement, we combine this lemma with a truncation argument for $N_g(p_{t,\epsilon/2})$.
Let $\epsilon > 0$ be fixed and let $M= \left\lceil\frac{1}{\epsilon^{k+2}}\right\rceil$.
To obtain an upper bound, fix $b \geq 0$.
Then,
\begin{align*}
&\P\left(\sum_{j=1}^{N_{g}(p_{t,\epsilon/2})} C_{j,t} \geq b\right) \leq
\P\left(\sum_{j=1}^{N_{g}(p_{t,\epsilon/2})} C_{j,t} \geq b, \ N_{g}(p_{t,\epsilon/2}) <  M \right) + \P(N_{g}(p_{t,\epsilon/2}) \geq  M )
\\
&
\leq
\P\left(\sum_{j=1}^{M} C_{j,t} \geq b \right) + \P(N_{g}(p_{t,\epsilon/2}) \geq  M)
\leq
\P\left(\sum_{j=1}^{M} C_{j,t} \geq b \right) + \sum_{k =  M }^{\infty}(1-p_{t,\epsilon/2})^{k}p_{t,\epsilon/2}
\\
&
\leq
\underbrace{\P\left(\sum_{j=1}^{M} C_{j,t} \geq b, \ \max_{1 \leq j \leq M}C_{j,t} \leq b(1-\epsilon_1)  \right) }_{\triangleq (I)}
+
\underbrace{\P\left(\max_{1 \leq j \leq M} C_{j,t} \geq b(1-\epsilon_1)\right)}_{ \triangleq (II)}
+
(1-p_{t,\epsilon/2})^{M}
\end{align*}

We start with the log-asymptotics of $(1-p_{t,\epsilon/2})^M$. Recall that $X_{t,\epsilon/2}$ is the solution of the following stochastic differential equation:
$
\mathrm{d}X_{t,\epsilon/2}(u)
=
u_{\epsilon/2}\left( X_{t,\epsilon/2}(u)\right)\mathrm{d}u
	+
\frac{2^{-1/2}\sigma}{t^{\gamma/2}} \mathrm{d} B(u), \ u \geq 0.
$
Hence, in view of Theorem~6.17 in \cite{klebaner2012introduction}, we can write for some normalizing constant $c$,
\begin{align*}
 1-p_{t,\epsilon/2} = \P_{\frac{\epsilon}{2}}
(\tau_{t,\epsilon/2}(\epsilon) < \tau_{t,\epsilon/2}(0))
=
\frac{\int_{0}^{\frac{\epsilon}{2}}\exp\left[{2 t^{\gamma}\int_{c}^{u}\frac{y}{(\epsilon/2)^{1-\kappa}} \mathrm{d}y}\right]\mathrm{d}u}
{\int_{0}^{\epsilon/2}\exp\left[{2 t^{\gamma}\int_{c}^{u}\frac{y}{(\epsilon/2)^{1-\kappa}} \mathrm{d}y}\right]\mathrm{d}u
+
\int_{\epsilon/2}^{\epsilon}\exp\left[{2 t^{\gamma}\int_{c}^{u}y^\kappa \mathrm{d}y}\right]\mathrm{d}u
}. \nonumber
\end{align*}
We can bound the second integral in the denominator using $y^\kappa\geq y/(\epsilon/2)^{1-\kappa}$ for $y<\epsilon$. Simplifying the resulting formula
we obtain after a straightforward computation,

\begin{align}
&\frac{\int_{0}^{\frac{\epsilon}{2}}
\exp\left[{2 t^{\gamma}
\big(\frac{u^2}{(\frac{\epsilon}{2})^{1-\kappa}}-(\frac{\epsilon}{2})^{2}\frac{1}{(\frac{\epsilon}{2})^{1-\kappa}} \big)}\right]\mathrm{d}u}
{\int_{0}^{\epsilon/2}\exp\left[{2 t^{\gamma}
\Bigg(
\frac{u^2}{(\frac{\epsilon}{2})^{1-\kappa}}}
-
(\frac{\epsilon}{2})^{2}\frac{1}{(\frac{\epsilon}{2})^{1-\kappa}}\Bigg)\right]\mathrm{d}u+\int_{\epsilon/2}^{\epsilon}\exp\left[{2 t^{\gamma}\Bigg(\frac{u^2}{(\frac{\epsilon}{2})^{1-\kappa}} }
-
(\frac{\epsilon}{2})^{2}\frac{1}{(\frac{\epsilon}{2})^{1-\kappa}}\Bigg)\right]\mathrm{d}u}
\nonumber
\\
%
&
\leq
\frac{\epsilon/2}
{
\int_{\epsilon/2}^{\epsilon}u\exp\left[{2 t^{\gamma}\Bigg(\frac{u^2}{(\frac{\epsilon}{2})^{1-\kappa}} }
-
(\frac{\epsilon}{2})^{2}\frac{1}{(\frac{\epsilon}{2})^{1-\kappa}}\Bigg)\right]\mathrm{d}u} \label{computation-exit-probability}
=
4{t^\gamma}(\epsilon/2)^{k} \cdot \exp\left[-\frac{3 t^\gamma \epsilon^{1+\kappa}}{2^\kappa}\right] \cdot \frac{1}{1- \exp\left[-{\frac{2 t^\gamma 3\epsilon^2}{4(\epsilon/2)^{1-\kappa}}}\right]  },
\end{align}
where in the inequality, we have bounded the integrand in the numerator by $1$, removed the left integral in the denominator, and lower-bounded the second integral by adding a factor $u<1$ to the integrand.

Due to \eqref{computation-exit-probability},
\begin{equation}\label{log-asy-exit-probability}
\limsup_{ t \to \infty}\frac{1}{t^\gamma} \log \P_{\epsilon/2}\left(\tau_{t,\epsilon/2}(\epsilon) < \tau_{t,\epsilon/2}(0)\right)^M \leq - \frac{3M}{2^\kappa}\epsilon^{1+\kappa}.
\end{equation}

We continue with term $(I)$.
To this end, for each $j=1,\ldots,M$, let $i_j$ have values in $\{0,\ldots,N(\epsilon_0)\}$ and consider the following partition:
\begin{align*}
&
D_N^{(j)}(\epsilon_0)
=
\left\{
0=y_{0}
\leq
\ldots
\leq
y_{N(\epsilon_0)}
=
b\cdot(1-\epsilon_1): y_{i_j+1}-y_{i_j} \leq \epsilon_0, i_j=0,\ldots,N(\epsilon_0)-1 \right\}.
\end{align*}
Based on the above partition, we derive the following upper bound:
\begin{align}\label{one-crude-inequality}
&
\P\left( \sum_{j=1}^{M} C_{j,t} \geq b, \ \max_{1 \leq j \leq M} C_{j,t} \leq b (1-\epsilon_1)\right)
\nonumber
\\
&
=
\int_{u_{M} \in [0,b(1-\epsilon_1)]}\ldots \int_{u_1 \in [0,b(1-\epsilon_1)]}\mathbf{I}\left\{ u_1+\ldots+u_{M} \geq b\right\} \prod_{j=1}^{M} \P\left(C_{j,t} \in \mathrm{d} u_j \right)
\nonumber
\\
&
\leq
\sum_{i_1=0}^{N(\epsilon_0)-1}
\ldots
\sum_{i_M=0}^{N(\epsilon_0)-1}
\mathbf{I}\left\{ y_{i_1+1}+\ldots+y_{i_M+1} \geq b
\right\}
\prod_{j=1}^{M}\P\left( y_{i_j} \leq C_{j,t}  < y_{i_j+1} \right)
\nonumber
\\
&
\leq
\sum_{i_1=0}^{N(\epsilon_0)-1}
\ldots
\sum_{i_M=0}^{N(\epsilon_0)-1}
\mathbf{I}\left\{ y_{{i}_1+1}+\ldots+y_{{i}_M+1} \geq b
\right\}
\prod_{j=1}^{M}\Big[\P\left( C_{j,t} \geq y_{i_j}\right)-\P\left( C_{j,t} \geq y_{{i}_j+1} \right)\Big]
\nonumber
\\
&
\leq
\sum_{i_1=0}^{N(\epsilon_0)-1}
\ldots
\sum_{i_M=0}^{N(\epsilon_0)-1}
\mathbf{I}\left\{ y_{{i}_1+1}+\ldots+y_{{i}_M+1} \geq b
\right\}
\prod_{j=1}^{M} \P\left( C_{j,t} \geq y_{i_j}\right).
\end{align}
Thanks to Lemma~\ref{tail-asymptotics-for-semicycle}, for any $\epsilon_2>0$ there exists $t_{\epsilon_2}$ such that for every $t \geq t_{\epsilon_2}$,
\begin{equation}\label{upper-bound-from-lemma36}
\P\left(C_{j,t} \geq y_{i_j} \right) \leq e^{-\big({y_{i_j}}\big)^\gamma t^\gamma \left(\mathrm{V}(0,\infty,\mathbf{D}, 1)-\epsilon_2 \right)}, \ \text{for each} \ 0 \leq  i_j \leq N(\epsilon_0), \ \text{and} \ j=1,\ldots,M.
\end{equation}
In view of \eqref{upper-bound-from-lemma36}, we obtain the following upper bound for \eqref{one-crude-inequality}:
\begin{align*}
&
\sum_{i_1=0}^{N(\epsilon_0)-1}
\ldots
\sum_{i_M=0}^{N(\epsilon_0)-1}
\mathbf{I}\left\{ y_{{i}_1+1}+\ldots+y_{{i}_M+1} \geq b
\right\}
\prod_{j=1}^{M}  e^{-\big({y_{i_j}}\big)^\gamma t^\gamma (\mathrm{V}(0,\infty,\mathbf{D}, 1)-\epsilon_2)} \nonumber
\\
&
\leq
\sum_{i_1=0}^{N(\epsilon_0)-1}
\ldots
\sum_{i_M=0}^{N(\epsilon_0)-1}
\mathbf{I}\left\{ y_{{i}_1+1}+\ldots+y_{{i}_M+1} \geq b
\right\}
\prod_{j=1}^{M}  e^{-\left(y_{{i}_j+1}-\epsilon_0\right)^\gamma t^\gamma (\mathrm{V}(0,\infty,\mathbf{D}, 1)-\epsilon_2)} \nonumber
\\
&
\leq
\int_{0}^{b(1-\epsilon_1)}
\ldots
\int_{0}^{b(1-\epsilon_1)}
\mathbf{I}_{\left\{ \sum_{i=1}^{M}u_i \geq b\right\}}
\prod_{i=1}^{M}  e^{\left(-(u_i)^\gamma+(\epsilon_0)^\gamma \right) t^\gamma (\mathrm{V}(0,\infty,\mathbf{D}, 1)-\epsilon_2)} \mathrm{d}u_1\ldots \mathrm{d}u_M.
\nonumber
\end{align*}
Next, observe that the last inequality above is bounded by
\begin{align}\label{crude-inequality-with-ld-estimates}
&
b^M\cdot
\max_{ \forall i, u_i \in \R_+
}
\left\{
\mathbf{I}_{\left\{ \sum_{i=1}^{M}u_i \geq b\right\}}
\cdot
\prod_{i=1}^{M}
\mathbf{I}_
{\left\{
u_i \leq b(1-\epsilon_1)
\right\}}
\cdot
e^{\left(-(u_i)^\gamma+(\epsilon_0)^\gamma \right) t^\gamma (\mathrm{V}(0,\infty,\mathbf{D}, 1)-\epsilon_2)} \right\}.
\end{align}
Furthermore, observe that this optimization problem is equivalent to \begin{equation}\label{sub-additive-finite-d-opt}
\min_{u_i \in \R_+}\left\{\sum_{i=1}^{M}u_i^\gamma: \ \sum_{i=1}^{M}u_i \geq b, \ \& \ u_i \leq b(1-\epsilon_1) \ \text{for every} \ i=1,\ldots,M\right\}.
\end{equation}
Due to the sub-additivity of the function $x \mapsto x^{\gamma}$, the optimal solution of \eqref{sub-additive-finite-d-opt} is attained at $(b(1-\epsilon_1), b\epsilon_1)$, and the optimal value is equal to $(b(1-\epsilon_1))^\gamma+ (b\epsilon_1)^\gamma$ (see (4.5) of \cite{bazhba2017sample}).
Hence, \eqref{crude-inequality-with-ld-estimates} is bounded by \[
b^M \cdot
\exp\Big[-{\Big((b(1-\epsilon_1))^\gamma+ (b\epsilon_1)^\gamma -M (\epsilon_0)^\gamma\Big)}t^\gamma (\mathrm{V}(0,\infty,\mathbf{D},1)-\epsilon_2)\Big].
\]
Therefore,
$
\limsup_{ t \to \infty}t^{-\gamma} \log(I) \leq
-\left[(b(1-\epsilon_1))^\gamma+ (b\epsilon_1)^\gamma-M (\epsilon_0)^\gamma \right] (\mathrm{V}(0,\infty,\mathbf{D},1)-\epsilon_2).
$
Since $\epsilon_2$ and $\epsilon_0$ are arbitrary, we can let $\epsilon_2$ and $\epsilon_0$ tend to $0$. That is,
\begin{equation}\label{log-Asy-for-term-II-1par}
\limsup_{ t \to \infty}t^{-\gamma} \log(I) \leq
-\left[(b(1-\epsilon_1))^\gamma+ (b\epsilon_1)^\gamma\right] \cdot \mathrm{V}(0,\infty,\mathbf{D},1).
\end{equation}

For term $(II)$, the union bound and the upper tail bound for $C_{1,t}$ imply that
\begin{align}\label{log-asympotics-term-2-tail-probability}
\limsup_{ t \to \infty}t^{-\gamma} \log(II)
&
\leq
\limsup_{ t \to \infty}t^{-\gamma} \log\left[M\P\left(C_{1,t} \geq b(1-\epsilon_1)\right)\right] 
\leq
-(b(1-\epsilon_1))^\gamma \cdot \mathrm{V}(0,\infty,\mathbf{D},1).
\end{align}
Recall that $M= \left\lceil\frac{1}{\epsilon^{k+2}}\right\rceil$.  Invoking \eqref{log-asy-exit-probability}, \eqref{log-Asy-for-term-II-1par}, \eqref{log-asympotics-term-2-tail-probability} and the principle of the largest term, we obtain
\begin{align*}
&
\limsup_{ t \to \infty}t^{-\gamma} \log\P\left(\sum_{j=1}^{N_{g}(p_{t,\epsilon/2})} C_{j,t} \geq b\right)
\nonumber
\\
&
\leq
\lim_{\epsilon_1 \downarrow 0}\lim_{\epsilon \downarrow 0}\max\left\{-\left[(b(1-\epsilon_1))^\gamma+ (b\epsilon_1)^\gamma\right] \cdot \mathrm{V}(0,\infty,\mathbf{D},1), -(b(1-\epsilon_1))^\gamma \cdot \mathrm{V}(0,\infty,\mathbf{D},1) \right\}
\nonumber
\\
&
\qquad\qquad\qquad\qquad \vee
\left(
-\left\lceil\frac{1}{\epsilon^{k+2}}\right\rceil\frac{3}{2^\kappa}\epsilon^{1+\kappa} \nonumber
\right)
\\
&
=-b^\gamma \cdot \mathrm{V}(0,\infty,\mathbf{D},1).
\end{align*}

For the lower bound, due to Lemma~\ref{tail-asymptotics-for-semicycle},
\begin{align*}
\lim_{\epsilon \downarrow 0}\lim_{t \to \infty}\frac{1}{t^{\gamma}}\log\P\left(\sum_{j=1}^{N_{g}(p_{t,\epsilon/2})} C_{j,t} \geq b\right)
&
\geq
\lim_{\epsilon \downarrow 0}\lim_{t \to \infty}\frac{1}{t^{\gamma}}\log\P\left( C_{1,t} \geq b\right)
\geq
-b^{\gamma} \cdot \mathrm{V}(0,\infty,\mathbf{D}, 1).
\end{align*}
\end{proof}

\subsection{Lower Bound for the Full Cycle $C_1^\delta$}\label{lower-bound-lemma2.1}

In this subsection, we derive the large deviations lower bound for the full cycle $C_1^{\delta}$.
To this end, we follow a similar strategy as for the large deviations upper bound.
First, we construct a modified diffusion process $\tilde{X}_{t,\epsilon}$ with a
Lipschitz drift and of which the area during its regeneration cycle now lower-bounds $C_1^{\delta}$.
We impose extra conditions, and we derive a lower bound for the area of  $\tilde{X}_{t,\epsilon}$ (Lemma~\ref{lower-bound-with-modified-process}).
Next, we prove that the variational problem associated with the lower bound has the same optimal value as the variational problem associated with the upper bound (Lemma~\ref{logarithim-lower-bound-for-full-cycle}).

Recall that the scaled diffusion process $X_t$ does not satisfy the framework of Result~\ref{FW-SPLDP}.
Therefore, we introduce the modified diffusion process $\tilde{X}_{t,\epsilon}$, as follows.
Let $\tilde{X}_{t,\epsilon}$ be the solution of the following SDE:
$
\mathrm{d}\tilde{X}_{t,\epsilon}(s)=\tilde{u}_{\epsilon}\big(\tilde{X}_{t,\epsilon}(s)\big)\mathrm{d}s+\sigma t^{-\gamma/2} \mathrm{d}B(s)
$, $s\geq 0$.
Since the drift term $\tilde{u}_{\epsilon}$ (defined in \eqref{eq-modifieddrift})  is a Lipschitz continuous function, the diffusion process $\tilde{X}_{t,\epsilon}$ satisfies the assumptions of Result~\ref{FW-SPLDP}, and hence, it satisfies the sample-path LDP in the space of continuous functions---over the time domain $[0,T]$---with rate function
$I_{x_0,\tilde{u}_{\epsilon}}^{T}(\xi)$.
Furthermore, let $\tilde{\tau}_{t,\epsilon}(0)=\inf\{s \geq 0: \tilde{X}_{t,\epsilon}(s)=0\}$ be the first hitting time to $0$ of the modified diffusion process $\tilde{X}_{t,\epsilon}$.
\begin{lemma}\label{lower-bound-with-modified-process}
Let $\delta > 0$, $\epsilon >0$, $z>0$, $H>0$ and $M>0$ be given.
Then, for every fixed $p>2\kappa$ and $t>0$,
	\begin{align*}
&
\P_{0}\left(
\int_{0}^{B_1^{\delta}}|X(s)|^p\mathrm{d}s > t
\right)
\\
&
\geq
\P_{\frac{\delta}{t^{\alpha/p}}}\left(
\tilde{X}_{t,\epsilon}(s) > s \ \mathrm{over} \ (0,z)
\right)
\P_{z}
\left(
\int_{0}^{H}\tilde{X}_{t,\epsilon}(s)\mathrm{d}s > 1, \
\inf_{s \in [0,H]}\{\tilde{X}_{t,\epsilon}(s)\} > z/2,  \ \sup_{s \in [0,H]}\{\tilde{X}_{t,\epsilon}(s)\} < M+z
\right).
\end{align*}
\end{lemma}

\begin{proof}
First, note that
	\begin{align*}
	\P_0\left(
	\int_{0}^{B_1^{\delta}}|X(s)|^p\mathrm{d}s > t
	\right) 	
 &\geq \P_0\left(
	\int_{A_1^{\delta}}^{B_1^{\delta}}|X(s)|^p\mathrm{d}s > t
	\right)
		=
	\P_{{\frac{\delta}{t^{\alpha/p}}}}\left(
	\int_{0}^{\tau_t(0) }|X_t(s)|^p\mathrm{d}s > 1
	\right),
\end{align*}
which in view of stochastic dominance is bigger than
\begin{align*}
	&
	\P_{{\frac{\delta}{t^{\alpha/p}}}}\left(
	\int_{0}^{\tilde{\tau}_{t,\epsilon}(0)}|\tilde{X}_{t,\epsilon}(s)|^p\mathrm{d}s > 1
	\right)
	\geq
	\P_{{\frac{\delta}{t^{\alpha/p}}}}\left(
	\int_{0}^{H+z}|\tilde{X}_{t,\epsilon}(s)|^p\mathrm{d}s > 1, \ \inf_{s \in [0,H+z]}\{\tilde{X}_{t,\epsilon}(s)\}>0
	\right)
	\\
	&
	\geq
	\P_{{\frac{\delta}{t^{\alpha/p}}}}\left(
	\int_{0}^{H+z}|\tilde{X}_{t,\epsilon}(s)|^p\mathrm{d}s > 1, \ \inf_{s \in [0,H+z]}\{\tilde{X}_{t,\epsilon}(s)\}>0, \ \tilde{X}_{t,\epsilon}(s) \geq s \ \text{over} \ [0,z]
	\right)
	\\
	&
	\geq
	\P_{\tilde{X}_{t,\epsilon}(z)}
	\left(
	\int_{z}^{H+z}|\tilde{X}_{t,\epsilon}(s)|^p\mathrm{d}s > 1, \
	\inf_{s \in [0,H+z]}\{\tilde{X}_{t,\epsilon}(s)\}>0 \Big|  \tilde{X}_{t,\epsilon}(s) \geq s \ \text{over} \ [0,z]\right)
	\\
	&
	\qquad\qquad \cdot
	\P_{{\frac{\delta}{t^{\alpha/p}}}}\left(\tilde{X}_{t,\epsilon}(s) \geq s \ \text{over} \ [0,z] \right)
	\\
	&
	\geq
	\P_{z}\left(
	\int_{0}^{H}|\tilde{X}_{t,\epsilon}(s)|^p\mathrm{d}s > 1, \ \inf_{s \in [0,H]}\{\tilde{X}_{t,\epsilon}(s)\}>0\right) \P_{\frac{\delta}{t^{\alpha/p}}}\left( \ \tilde{X}_{t,\epsilon}(s) \geq s \ \text{over} \ [0,z]
	\right)
	\\
	&
	\geq
	\P_{{\frac{\delta}{t^{\alpha/p}}}}\left(
\tilde{X}_{t,\epsilon}(s) > s \ \text{over} \ (0,z)
\right)
\P_{z}
\left(
\int_{0}^{H}\tilde{X}_{t,\epsilon}(s)\mathrm{d}s > 1, \
\inf_{s \in [0,H]}\{\tilde{X}_{t,\epsilon}(s)\} > z/2,  \ \sup_{s \in [0,H]}\{\tilde{X}_{t,\epsilon}(s)\} < M+z
\right),
	\end{align*}
	where in the first equality as well as the second last inequality we have used the strong Markov property.
\end{proof}
Define the sets $I_{H,\epsilon/2}=\{	\inf_{s \in [0,H]}\{\tilde{X}_{t,\epsilon}(s)\} > \epsilon/2\}$ and $S^M_{H,\epsilon}=\{\sup_{s \in [0,H]}\{\tilde{X}_{t,\epsilon}(s)\} < M+\epsilon\} $.

\begin{lemma}\label{logarithim-lower-bound-for-full-cycle}
For every fixed $p>2\kappa$ and $\delta>0$, it holds that	
	\begin{itemize}
		\item[$i$)]
		$
		\lim_{\epsilon \downarrow 0}\liminf_{t \to \infty}t^{-\gamma}\log\P_{\frac{\delta}{t^{\alpha/p}}}\left(
		\tilde{X}_{t,\epsilon}(s) > s \ \mathrm{for\ all}\ s \ \in \ (0,\epsilon)
		\right) = 0.
		$
		\item[$ii$)] In addition,
		\begin{align*}
&\lim_{H \to \infty}\lim_{M \to \infty} \lim_{\epsilon \downarrow 0}\liminf_{t \to \infty}t^{-\gamma}\log\P_{\epsilon}
		\left(
		\int_{0}^{H}(\tilde{X}_{t,\epsilon}(s))^p\mathrm{d}s > 1, \tilde{X}_{t,\epsilon}(s) \in [\epsilon/2, M+\epsilon], s\in [0,H]
		\right)
		\geq -\mathrm{V}(0,\infty,\mathbf{D},1).
		\end{align*}
	\end{itemize}
\end{lemma}
\begin{proof}
	We start with \textit{i}).
	Note that $\tilde{X}_{t,\epsilon}$ s.t. $\tilde{X}_{t,\epsilon}(0)=\frac{\delta}{t^{\alpha/p}}$ is exponentially equivalent to $\tilde{X}_{t,\epsilon}$ initialized at $0$.
	That is,
	$
\P_{\frac{\delta}{t^{\alpha/p}}}\left(
	\tilde{X}_{t,\epsilon}(s) > s \ \text{for all} \ s \in \ (0,\epsilon)
	\right)$
	satisfies the same log-asymptotics as $
\P_{0}\left(
\tilde{X}_{t,\epsilon}(s) > s \ \text{for all} \ s \in \ (0,\epsilon)
	\right)
	$ as $t$ tends to infinity.
	Now, let $B=\{\xi \in \mathbb{C}[0,z]: \xi(0)=0, \ \xi(s) > s \ \text{over} \ (0,\epsilon) \}$, and set $\xi_0=2s$.
In view of the sample-path LDP of $\tilde{X}_{t,\epsilon}$,
	$
	\liminf_{t \to \infty}t^{-\gamma}\log\P_{0}\left(
	\tilde{X}_{t,\epsilon}(s) > s \ \text{over} \ (0,\epsilon)
	\right)
	\geq
	-\inf_{\xi \in B}I_{0,\tilde{u}_{\epsilon}}^{\epsilon}(\xi)
	\geq
	-I_{0,\tilde{u}_{\epsilon}}^{\epsilon}(\xi_0)
	$.
	Since 	$
	I_{0,\tilde{u}_{\epsilon}}^{\epsilon}(\xi_0) $ is bounded by $ \frac{1}{\sigma^2}\int_{0}^{\epsilon}(2+2\epsilon^{\kappa})^2 \mathrm{d}s=\frac{\epsilon}{\sigma^2}(2+2\epsilon^{\kappa})^2
	$, the desired inequality is obtained by letting $\epsilon$ tend to $0$.

\textit{Proof of ii}). Recall that
 $\int_{0}^{H}\xi(s)^p\mathrm{d}s$ is a continuous functional w.r.t. the supremum norm topology, and
 the diffusion process $\tilde{X}_{t,\epsilon}$ satisfies the conditions of Result~\ref{FW-SPLDP}.
  Hence, over fixed boundaries, the area swept under the trajectory of $\tilde{X}_{t,\epsilon}$ satisfies the sample-path LDP in $\mathbb{C}[0,H]$.
Thanks to the contraction principle, we obtain
\begin{align*}
&
\liminf_{t \to \infty}t^{-\gamma}\log\P_{\epsilon}
		\left(
		\int_{0}^{H}|\tilde{X}_{t,\epsilon}(s)|^p\mathrm{d}s > 1, \tilde{X}_{t,\epsilon}(s) \in [\epsilon/2, M+\epsilon], s\in [0,H]
		\right)
		\geq
		-\inf_{\xi \in A_{[\epsilon/2,M+\epsilon]}(\epsilon,H,1)}I_{\epsilon,\tilde{u}_{\epsilon}}^{H}(\xi).		
\end{align*}

In view of Lemma~\ref{continuity-properties-V-2}---part \textit{i}), we set $y=\epsilon/2$, $x=\epsilon$, and we obtain
	\[
	\inf_{\xi \in A_{[\epsilon/2,M+\epsilon]}(\epsilon,H,1)}I_{\epsilon,\tilde{u}_{\epsilon}}^{H}(\xi)-\frac{2\epsilon^\kappa}{\sigma^2}\left(M+M^\kappa H+{\epsilon^{\kappa}} H \right)
		 < \inf_{\xi \in A_{[0,M]}(0,H,1)}I_{0,\mathbf{D}}^{H}(\xi).
		 \]
Hence,
$-\lim_{\epsilon \downarrow 0}\inf_{\xi \in A_{[\epsilon/2,M+\epsilon]}(\epsilon,H,1)}I_{\epsilon,\tilde{u}_{\epsilon}}^{H}(\xi) \geq -\inf_{\xi \in A_{[0,M]}(0,H,1)}I_{0,\mathbf{D}}^{H}(\xi).$
Finally, we let $M$ and $H$ tend to $\infty$, therefore, in view of \textit{ii}) of Lemma~\ref{continuity-properties-V-2},
$$\lim_{H \to \infty}\lim_{M \to \infty}\inf_{\xi \in A_{[0,M]}(0,H,1)}I_{0,\mathbf{D}}^{H}(\xi)=\mathrm{V}(0, \infty,\mathbf{D},1).$$

\end{proof}

\subsection{Proof of Lemma~\ref{semi-exponential-property-cycles}}\label{most-important-block}
\label{proof-for-ld-estimates-c1d}

In this subsection, we conclude the proof of Lemma~\ref{semi-exponential-property-cycles}.

\begin{proof}
It suffices to consider $b=1$.
Parts \textit{i}) and \textit{ii}) of Lemma~\ref{distributional-equalities-via-scaled-diffusion-process} imply that, for every fixed $\epsilon_0>0$ and sufficiently large $t$,
		\[
		\P_0\left(C_1^\delta \geq t \right)=\P_0\left(\int_{0}^{\tau_t(0)}|X_{t}(s)|^{p} \mathrm{d}s > 1\right)
		\leq
		\P_{\frac{\delta}{t^{\alpha/p}}}\left(\int_{0}^{\tau_t(0)}|X_{t}(s)|^{p} \mathrm{d}s > 1-\epsilon_0\right)
		+
	\bigO\left(\exp\left( -t\epsilon_{0}{\frac{\delta^{2\kappa-p}}{\sigma^2}}\right)\right).
		\]
Combining this with
Lemma~\ref{stochastic-upper-bound-via-modified-dif-n-iid-cjs}, we have
\begin{align*}
	\P_0\left(C_1^\delta \geq t \right)
	&
	\leq
	\underbrace{	\bigO\left(\exp\left( -t\epsilon_{0}{\frac{\delta^{2\kappa-p}}{2\sigma^2}}\right)\right)}_{\triangleq (I)}
	+
	2\underbrace{\P_{\frac{\delta}{t^{\alpha/p}}}
		\left(
		\int_{0}^{\tau_{t}(0)}(X_{t}(s))^p \cdot \mathbf{I}\left(X_{t}(s) \leq \epsilon \right) \mathrm{d}s \geq \epsilon_0(1-\epsilon_0)
		\right)}_{ \triangleq (II)}
	\\
	&
	\qquad\qquad\qquad\qquad\qquad\quad	
	+ 2\underbrace{\P\left(
		\sum_{j=1}^{N_{g}(p_{t,\epsilon/2})} C_{j,t} \geq (1-\epsilon_0)^2
		\right)}_{\triangleq (III)}.
	\end{align*}
Invoking Lemma~\ref{area-for-epsilon-bounded-trajectories},  Lemma~\ref{log-asy-sum-of-eps-semi-cycles}, and applying the principle of the largest term, we obtain
	\begin{align*}
		\limsup_{t \to \infty}\frac{1}{t^{\gamma}}\log \P_0\left(C_1^{\delta} \geq t \right)
	&
	\leq
	\lim_{\epsilon \downarrow 0}\max\left\{
	\lim_{t \to \infty}\frac{\log(I)}{t^{\gamma}}, \
	\lim_{t \to \infty}\frac{\log(II)}{t^{\gamma}}, \
\limsup_{t \to \infty}\frac{\log(III)}{t^{\gamma}}\right\}
	\\
	&
=	\limsup_{t \to \infty}\frac{\log(III)}{t^{\gamma}}
	\leq
	-(1-\epsilon_0)^{2\gamma}
	\cdot
	\mathrm{V}(0,\infty,\mathbf{D}, 1
	).
	\end{align*}
	Since $\lim_{\epsilon_0 \downarrow 0}(1-\epsilon_0)^{2\gamma}\cdot\mathrm{V}(0,\infty,\mathbf{D}, 1)
	=
	\mathrm{V}(0,\infty,\mathbf{D}, 1)$,
	we conclude the proof of the upper bound.

	For the lower bound,
	in view of Lemma~\ref{lower-bound-with-modified-process}, recall
	that
	$
	\mathrm{d}\tilde{X}_{t,\epsilon}(s)=\tilde{u}_{\epsilon}\big(\tilde{X}_{t,\epsilon}(s)\big)\mathrm{d}s+\frac{2^{-1/2}\sigma}{t^{\gamma/2}} \mathrm{d}B(s)$, $s\geq 0$.
 Hence,
	\begin{align*}
		&\P_{0}\left(
C_1^\delta > t
	\right)
	=
	\P_0\left(
	\int_{0}^{B_1^{\delta}}|X(s)|^p\mathrm{d}s > t
	\right)
	\\
	&\geq
	\P_{\frac{\delta}{t^{\alpha/p}}}\left(
	\tilde{X}_{t,\epsilon}(s) > s, \ \ s\in (0,\epsilon)
	\right)
	\P_{\epsilon}
	\left(
	\int_{0}^{H}\tilde{X}_{t,\epsilon}(s)\mathrm{d}s > 1, \
	\inf_{s \in [0,H]}\{\tilde{X}_{t,\epsilon}(s)\} > \epsilon/2,  \ \sup_{s \in [0,H]}\{\tilde{X}_{t,\epsilon}(s)\} < M+\epsilon
	\right).
	\end{align*}	
	Next, in view of \textit{i}) and \textit{ii}) of Lemma~\ref{logarithim-lower-bound-for-full-cycle},
	\begin{itemize}
		\item[$i$)]	$
		\lim_{\epsilon \downarrow 0}\liminf_{t \to \infty} t^{-\gamma} \log\P_{\frac{\delta}{t^{\alpha/p}}}\left(
		\tilde{X}_{t,\epsilon}(s) > s, \ \ s\in  (0,\epsilon)
		\right) = 0.
		$
		\item[$ii$)] For $I_{H,\epsilon/2}=\{	\inf_{s \in [0,H]}\{\tilde{X}_{t,\epsilon}(s)\} > \epsilon/2\}$ and $S^M_{H,\epsilon}=\{\sup_{s \in [0,H]}\{\tilde{X}_{t,\epsilon}(s)\} < M+\epsilon\} $,
		\begin{align*}
		\lim_{H \to \infty}\lim_{M \to \infty}\lim_{\epsilon \downarrow 0}\liminf_{t \to \infty} t^{ -\gamma} \log\P_{\epsilon}
		\left(
		\int_{0}^{H}|\tilde{X}_{t,\epsilon}(s)|^p\mathrm{d}s > 1, \
		I_{H,\epsilon/2}, \ S^M_{H,\epsilon}
		\right)
		\geq -\mathrm{V}(0,\infty,\mathbf{D},1).
		\end{align*}
	\end{itemize}
	Consequently,
	\begin{align*}
	\liminf_{t \to \infty}{t^{-\gamma}}
	\log \P\left(
C_1^\delta > t
	\right)
	\geq
	-\mathrm{V}(0,\infty,\mathbf{D},1).
	\end{align*}
	Since the upper and lower bounds coincide, the tail asymptotics of $C_1^\delta$ are established.
\end{proof}

\section{Proof of Lemma~\ref{bounded-moment-generating-function-t-1-delta}}
\label{section-proof-of-ndelta-asy}
		Recall that $\tau_1^\delta = B_1^\delta$ and $B_1^\delta = \inf\{t > A_1^{\delta} : |X(t)| = 0 \}$.
		In view of the strong Markov property, $
	\inf\{t > A_1^\delta: |X(t)|=0\} \overset{\cal{D}}{=}A_1^\delta+\inf\{s > 0: |X(s)|=0 \big| X(0)=\delta\}$, hence, it suffices to prove that the exponential moments of $A_1^\delta$ and $B_1^\delta|X(0)=\delta$ are finite in a neighborhood around $0$.
	Towards this, we invoke an upper bound for densities of first passage  times of diffusion processes (Result~\ref{bound-hitting-times-of-diffusion-processes}).
	However, $X$ does not satisfy the conditions of Result~\ref{bound-hitting-times-of-diffusion-processes} since we cannot ensure the existence of a strong solution.
	To deal with this, we construct the hitting times $\tau_\epsilon(\delta)$ and $\tau_{\epsilon}(0)$ that serve as stochastic upper bounds for $A_1^\delta$ and $B_1^\delta|X(0)=\delta$.
	Then, we verify that Result~\ref{bound-hitting-times-of-diffusion-processes} applies to $\tau_{\epsilon}(\delta)$ and $\tau_{\epsilon}(0)$, and we prove that the moment generating functions of $\tau_\epsilon(\delta)$ and  $\tau_{\epsilon}(0)$ are finite in a neighborhood around $0$.
	This implies the existence of  the moment generating function of $B_1^\delta$.
	
	We start with the moment generating function of $A_1^\delta$.
We construct a hitting time $\tau_\epsilon(\delta)$ based on a diffusion process $\check{X}_\epsilon$ that is stochastically smaller than $X$, and which has an absolutely continuous drift and admits a unique strong solution.
	Given $\epsilon > \delta / \sigma$, let $\check{X}_{\epsilon}(\cdot)$  be the solution of the following SDE:
	$$
	\mathrm{d}\check{X}_{\epsilon}(s)=\tilde u_\epsilon\left(\check{X}_{\epsilon}(s)\right)\mathrm{d}s+2^{-1/2}\sigma \mathrm{d}B(s), \ \mathrm{such\ that} \ \check{X}_{\epsilon}(0)=0.
	$$
	
	With regard to the diffusion process $\check{X}_{\epsilon}$, let $\tau_{\epsilon}(\delta)=\inf\{s \geq 0: \check{X}_{\epsilon}(s)=\delta\}$, i.e.,  the first time the diffusion $\check{X}_{\epsilon}$ reaches $\delta$.
	Since the drift $\tilde u_\epsilon$ is smaller than
	$\mathbf{D}$, by Theorem~1.1 in \cite{yamada1973comparison}, the diffusion process
	$\check{X}_{\epsilon}$ is stochastically smaller than the diffusion process $X$.
	As a result, $A_1^{\delta}$ is stochastically smaller than $\tau_\epsilon(\delta)$ in the usual stochastic order.
Consequently, $\E\left(e^{\theta A_1^{\delta}}\right)\leq \E\left(e^{\theta \tau_\epsilon(\delta)}\right)$.
	
	Now, we prove that $\E\left(e^{\theta \tau_\epsilon(\delta)}\right)$ is finite for some $\theta >0$.
	To apply Result~\ref{bound-hitting-times-of-diffusion-processes}, we have to work with a unit diffusion coefficient.
	Consider the function $F(y) \triangleq \frac{y}{2^{-1/2}\sigma}$.
	Due to Itô's formula, $\check{Y}_\epsilon(t) \triangleq F(\check{X}_{\epsilon}(t))$ satisfies
$
	\mathrm{d}\check{Y}_\epsilon(t)= l(\check{Y}_\epsilon(t))\mathrm{d}t+\mathrm{d}B(t), \   \check{Y}_\epsilon(0)=0,
$
	where
	$$
	l(x)
	=
	\begin{cases}
		-\frac{1}{(2^{-1/2}\sigma)^{1-\kappa}}x^{\kappa}, &  \ \text{for} \ x > \epsilon/(2^{-1/2}\sigma)
		\\
		-\epsilon^\kappa/(2^{-1/2}\sigma), &  \ \text{for} \ x \leq \epsilon/ (2^{-1/2}\sigma)
	\end{cases}.
	$$
	Then,
	\begin{align*}
	\tau_{\epsilon}(\delta)=\inf\{s \geq 0: \check{X}_{\epsilon}(s)=\delta\} &
	=\inf\{s \geq 0: F(\check{X}_{\epsilon}(s))=F(\delta)\}
	=\inf\{s \geq 0: \check{Y}_\epsilon(s)=\delta/(2^{-1/2}\sigma)\}.
	\end{align*}

	In view of Result~\ref{bound-hitting-times-of-diffusion-processes}, the first passage time $\tau_{\epsilon}(\delta)$ has a density $p_{\tau_{\epsilon}(\delta)}$ satisfying the following inequality:
	\begin{equation*}\label{upper-bound-for-density-x-epsilon}
	p_{\tau_\epsilon(\delta)}(t) \leq \frac{1}{t}(\delta/(2^{-1/2}\sigma)+\delta t)q(t,0,\delta/(2^{-1/2}\sigma))e^{G(\delta/(2^{-1/2}\sigma))-G(0)-\frac{t}{2}M(t)}, \qquad t \geq 0,
	\end{equation*}
	where
	$q(t,y,z)$ denotes the transition density of the Brownian motion,
	$G(y)=\int_{0}^{y}l(z)\mathrm{d}z$, and finally, we also have,
	$M(t) = \essinf_{y \in (-\infty,\delta/ (2^{-1/2}\sigma) )}\left\{l'(y)+l^2(y)\right\}$.
Since
	 $q\left(t,0,\frac{\delta}{2^{-1/2}\sigma}\right)=\frac{1}{\sqrt{2 \pi t}}\exp\left[-\frac{\delta^2}{\sigma^2  t}\right]$,
	 $G(0)=0$ and $M(t) \leq 2\epsilon^{2\kappa}/ \sigma^2$
we deduce the following upper bound:
	\begin{align*}
	  \E\left(e^{\theta \tau_\epsilon(\delta)}\right) 
	&
	\leq
	\int_{\R_+} e^{\theta t} \frac{1}{t}\left(\frac{\delta}{2^{-1/2}\sigma}+\delta t \right) \frac{1}{\sqrt{2 \pi t}}\exp\left[-\frac{\delta^2}{\sigma^2  t}+G(\delta/(2^{-1/2}\sigma))-t\cdot \epsilon^{2 \kappa}/\sigma^2\right] \mathrm{d}t
	\\
	&
	\leq
	\int_{\R_+} \frac{1}{t}\left(\frac{\delta}{2^{-1/2}\sigma}+\delta t \right) \exp{\left[\left(\theta-\frac{\epsilon^{2 \kappa}}{\sigma^2} \right)\cdot t\right]} \frac{1}{\sqrt{2 \pi t}}\exp\left[-\frac{\delta^2}{\sigma^2  t}+G(\delta/ (2^{-1/2}\sigma))\right] \mathrm{d}t.
	\end{align*}
	By the above upper bound, $\E(e^{\theta \tau_\epsilon(\delta)})$ is finite for all $\theta < \epsilon^{2\kappa}/2\sigma^2$, thus,  $\E(e^{\theta A_1^\delta})$ is finite for all $\theta < \epsilon^{2 \kappa}/2\sigma^2$.
	
	We apply a similar argument to the m.g.f. of $B_1^\delta|X(0)=\delta$.	
	To obtain a convenient framework, let $\hat {X}_{\epsilon}(\cdot)$ be the solution of the following SDE:
	$
	\mathrm{d}\hat{X}_{\epsilon}(s)=\mathbf{U}_\epsilon\left(\hat{X}_{\epsilon}(s)\right)\mathrm{d}s+\sigma \mathrm{d}B(s), \ \mathrm{such\ that}\ \hat{X}_{\epsilon}(0)=\delta,
	$
	where the drift term $\mathbf{U}_{\epsilon}$ is equal to
	\begin{equation*}
	\mathbf{U}_\epsilon(x)
	=
	\begin{cases}
	x^{\kappa}, &  \ \text{for} \ x > \epsilon
	\\
	\frac{x}{\epsilon^{1-\kappa}}, &  \ \text{for} \ 0 \leq x \leq \epsilon
	\\
	-sgn(x)|x|^{\kappa+1}, &  \ \text{for} \ -1 \leq x < 0
	\\
	-sgn(x)|x|^{\kappa}, &  \ \text{for} \ x < -1
	\end{cases}.
	\end{equation*}
	Furthermore, let $\tau_{\epsilon}(0)=\inf\{s \geq 0: \hat{X}_{\epsilon}(s)=0\}$. 
	The drift $\mathbf{U}_\epsilon$ is bigger than
	$\mathbf{D}$, therefore and similar to above, by Theorem~1.1 in \cite{yamada1973comparison}, the diffusion process
	$\hat{X}_\epsilon $ is stochastically larger than the diffusion process $X$.
 Consequently, $B_1^\delta \mid X(0)=\delta$ is stochastically smaller than $\tau_\epsilon(0)$ in the usual stochastic order
and $\E\left(e^{\theta B_1^\delta}\big|X(0)=\delta\right)\leq \E\left(e^{\theta \tau_\epsilon(0)}\right)$.

	It remains to show that $\E\left(e^{\theta \tau_\epsilon(0)}\right)$ is finite for some $\theta >0$.
	To apply Result~\ref{bound-hitting-times-of-diffusion-processes}, we again have to work with a unit diffusion coefficient.
	Thus, we consider the function $F(y) = \frac{y}{(2^{-1/2}\sigma)}$.
	By Itô's formula, $Y_\epsilon(t) \triangleq F(\hat{X}_{\epsilon}(t))$ satisfies
\[
\mathrm{d}Y_\epsilon(t)= v(Y_\epsilon(t))\mathrm{d}t+\mathrm{d}B(t), \qquad Y_\epsilon(0)=\delta/(2^{-1/2}\sigma),
\]	
where $v(y)=\frac{1}{2^{-1/2}\sigma}\mathbf{U}_{\epsilon}(2^{-1/2}\sigma y)$.
Therefore,
$\tau_{\epsilon}(0)
=\inf\{s \geq 0: Y_\epsilon(s)=0\}$
and its density $p_{\tau_{\epsilon}(0)}$ satisfies
	\begin{equation*}\label{upper-bound-for-density-x-epsilon-2}
	p_{\tau_\epsilon(0)}(t) \leq \frac{1}{t2^{-1/2}\sigma} \delta q\left(t,\delta/(2^{-1/2}\sigma),0\right)e^{G(0)-G(\delta/(2^{-1/2}\sigma))-\frac{t}{2}M^*(t)}, \qquad t \geq 0,
	\end{equation*}
	where
	$q(t,y,z)$ is the transition density of the Brownian motion, $G(y)=\int_{0}^{y}v(z)\mathrm{d}z$, and for every $t \geq 0$,
	$M^*(t) = \essinf_{y \in (0,\infty)}\left\{v'(y)+v^2(y)\right\}$.
With regard to each one of the terms above,
 $q\left(t,\frac{\delta}{2^{-1/2}\sigma},0\right)=\frac{1}{\sqrt{2 \pi t}}\exp\left[-\frac{\delta^2}{\sigma^2 t}\right]$,
 $M^*(t)=\essinf_{x \geq \epsilon/(2^{-1/2}\sigma)}\{\frac{1}{\sigma^{2(1-\kappa)}} x^{2\kappa}+\frac{\kappa}{(2^{-1/2}\sigma)^{(1-\kappa)}}{x^{1-\kappa}}\} \wedge \essinf_{ x \in [0,\epsilon/(2^{-1/2}\sigma)]}\{\frac{1}{\epsilon^{1-\kappa}}+\frac{x^2}{\epsilon^{2-2\kappa}}\} \triangleq \theta_\epsilon^*$, and
  $G(0)=0$.
	It can be verified that $\theta^*_{\epsilon}$ is strictly positive.
	Consequently,
	\begin{align*}
	\E\left(e^{\theta \tau_\epsilon(0)}\right)
	&
	\leq
	\int_{\R_+} e^{\theta t} \frac{1}{t 2^{-1/2}\sigma} \delta \frac{1}{\sqrt{2 \pi t}}\exp\left[-\frac{\delta^2}{\sigma^2 t}-G(\delta/(2^{-1/2}\sigma))-t\cdot\theta_\epsilon^*/2\right] \mathrm{d}t \nonumber
	\\
	&
	\leq
	\int_{\R_+} \frac{1}{t 2^{-1/2} \sigma} \delta \exp{\left[\left(\theta-\theta_\epsilon^*/2 \right)\cdot t\right]} \frac{1}{\sqrt{2 \pi t}}\exp\left[-\frac{\delta^2}{\sigma^2  t}-G(\delta/(2^{-1/2}\sigma))\right] \mathrm{d}t.
	\end{align*}
	Hence, due to the upper bound derived above, $\E(e^{\theta \tau_\epsilon(0)})$ is finite for $\theta < \theta_\epsilon^*/2$.
	Consequently, $\E(e^{\theta B_1^\delta}\big| X(0)=\delta)$ is finite for all $\theta < \theta^*_\epsilon/2$.
	This, along with $\E(e^{\theta A_1^\delta})$ being finite for $\theta$ smaller than $\epsilon^{2\kappa}/2\sigma^2$, implies that $\E(e^{\theta \tau_1^\delta})$ is finite for all $\theta < \frac{1}{2}\min(\epsilon^{2\kappa}/\sigma^2, \theta_\epsilon^*)$.

\begin{appendix}
\section{Useful  LDPs}

In this subsection, we first present a useful sample-path large deviations principle for diffusion processes.
Recall that $\mathbb{C}[0,T]$ is the space of continuous functions---over the domain $[0,T]$---and that $\mathscr{H}[0,T]$ is the subspace of absolutely continuous functions with square integrable derivative.
Let $Y_{\epsilon}$, $\epsilon>0$, satisfy the following stochastic differential equation:
\begin{align*}\label{eq:freidlin-wentzel-sde-system}
&
\mathrm{d}Y_{\epsilon}(t)=-\mathbf{b}(Y_{\epsilon})\mathrm{d}t+
\epsilon \cdot a \cdot \mathrm{d}B(t), \quad t \in [0,T], \quad  
Y_{\epsilon}(0)=x_0.
\end{align*}
Here, $B$ denotes a standard Brownian motion and $a$ is a positive constant.
The following result is a consequence of Theorem~5.6.7 and Exercise~5.6.24 in \cite{dembo2010large}:

\begin{result}\citep{dembo2010large}\label{FW-SPLDP}
	Suppose that $\mathbf{b}$ is a Lipschitz continuous function, then the stochastic process $Y_{\epsilon}$ satisfies the large deviations principle in $\mathbb{C}[0,T]$ with speed $1/\epsilon^2$ and with good rate function $I_{x_0, \mathbf{b}}^{T}: \mathbb{C}[0,T] \rightarrow \R_+ $, where
	\begin{equation*}\label{F.W.-rate-function}
	I_{x_0, \mathbf{b}}^{T}(\xi)
	=
	\begin{cases}
	\int_{0}^{T}\frac{|\dot{\xi}(s)-\mathbf{b}(\xi(s))|^2 }{2a^2}\mathrm{d}s, & \mathrm{for} \ \xi \in \mathscr{H}[0,T] \ \& \ \xi(0)=x_0
	\\
	\infty, & \mathrm{otherwise}
	\end{cases}.
	\end{equation*}
\end{result}

Next, we present a useful large deviations result for certain weighted random sums.

\begin{result}\label{LDP-semiexponential-rvs}\cite{gantert2014large}
Let $\{X_j\}_{j \geq 1}$ be a sequence of i.i.d. random variables on a probability space $(\Omega, \mathcal{F}, \P)$,
and let $m=\E(X_1)$.
Suppose that there exist a constant $r \in (0,1)$, slowly varying functions
$b, c_1, c_2: (0,\infty) \to (0,\infty)$ and a constant $t^* > 0$ such that for $t \geq t^*$,
\[
c_1(t)\exp(-b(t)t^r) \leq \P(X_1 \geq t) \leq c_2(t)\exp(-b(t)t^r).
\]
Moreover, for every $n \in \N$, let $\{a_i(n)\}_{i \in \N}$ be a sequence of non-negative numbers that satisfy the following assumptions:
\begin{itemize}
\item[$i$)] There exists a real number $s_1 \neq 0$ such that the sequence $\{R(1,n)\}_{n \in \N}$ of real numbers defined by $\sum_{i=1}^{n}a_i(n)=s_1 R(1,n)$, for all $n \in \N$, satisfies $R(1,n) \to 1$ as $n \to \infty$.
\item[$ii$)] There exists a real number $s$ such that, for $a_{max}(n) \triangleq \max_{1 \leq i \leq n}a_i(n)$,  $\lim_{n \to \infty}n a_{max}(n)=s$.
\end{itemize}
Consider the weighted random walk $\bar S_n =\sum_{i=1}^{n}a_i(n)X_i$.
Then,
\[
\lim_{n \to \infty}\frac{1}{b(n)n^r}\log \P\left( \bar S_n \geq x \right)=-\left(\frac{x}{s}-\frac{s_1}{s}m \right)^r, \qquad \forall x > s_1 m.
\]
\end{result}

\begin{remark}\label{weighted-random-walk}
In the case of $\bar S_{\left\lfloor nB \right\rfloor} =\frac{1}{n}\sum_{i=1}^{\left\lfloor nB \right\rfloor}X_i$, $B>0$, we obtain an LDP by simply resetting $nB=n$.
This resetting results in a weighted random walk with weights $a_i(n)=B/n$.
Consequently, $s_1=s=B$, $\P(\bar S_{\left\lfloor nB \right\rfloor} \geq x)=\P\left(\bar S_n \geq x/B \right)$, and hence,
 \[
\lim_{n \to \infty}\frac{1}{b(n)n^r}\log\P\left(\bar S_{\left\lfloor nB \right\rfloor} \geq x \right)
=
-\left(x-mB \right)^r, \qquad \forall x > B m.
\]
\end{remark}

We finally present two results required in the proof of Lemma~\ref{area-for-epsilon-bounded-trajectories}.
The first result is a special case of Theorem~1 in \cite{bryc1997large} and is instrumental for the case $\kappa=1$.

\begin{result}\label{dembobryc}
Consider the stationary OU process  $\mathrm{d}Y(t) = - aY(t)\,\mathrm{d}t + \sqrt{a}\,\mathrm{d}B(t)$, with $B$ a standard Brownian motion and $a$ a positive constant.
Define $S_T\triangleq\int_0^T Y_t^2 \,\mathrm{d}t$.
Then $T^{-1}S_T$ satisfies an LDP with speed $T$ and rate function $I(x) = (a/4) (\sqrt{x}-1/\sqrt{x})^2$.
\end{result}

For the case $\kappa<1$, we apply the following result, which is a special case of Theorem~2.1 in \cite{wuyiao2008}.
We introduce some concepts first, specialized to the real line.
A Young function is a left-continuous convex even function $\Phi$ such that $\Phi(0)=0$, $\Phi(r) \rightarrow\infty$ as $r\rightarrow\infty$, and $\Phi(r)<\infty$  for (at least) some $r>0$.
Given a probability measure $\mu$, the Orlicz space $L_\Phi(\mu)$ is the space of all real-valued measurable functions $u$ such that $\Phi(\alpha |u(\cdot)|)$ is $\mu$-integrable
for some $\alpha>0$.
Define the Orlicz norm
$$
||f||_\Phi \triangleq \sup_{g: \int \Psi(|g|)\,\mathrm{d}\mu \leq 1} \int |fg| \,\mathrm{d}\mu,
$$
with $\Psi$ the Legendre transformation of $\Phi$.

\begin{result}
\label{wuyiao}
Let $X$ be a stationary real-valued Markov process with generator $\mathcal{L}$ and invariant distribution $\mu$.
Assume there exists a Young function $\Phi$ and
real-valued constants $C_1$ and $C_2$ such that the functional inequality
\begin{equation}
\label{condition-wu1}
    ||f^2||_{\Phi} \leq C_1 \int -(\mathcal{L}f(x))f(x) \,\mathrm{d}\mu(x) + C_2 \int f(x)^2 \,\mathrm{d}\mu(x)
\end{equation}
holds for all twice-differentiable functions $f \in L_\Phi(\mu)$.
In addition, assume that $F$ is a real-valued function such that
\begin{equation}
\label{condition-wu2}
    \int \Psi(\lambda |F(x)|) \,\mathrm{d}\mu (x) <\infty, \quad \lambda \in (0,\infty).
\end{equation}
Then $\int_0^T F(X(t))\,\mathrm{d}t/T$ satisfies an LDP with speed $T$ and rate function with compact level sets.
\end{result}

\section{A Bound on First Passage Time Densities}

Let $X$ satisfy the following SDE: $\mathrm{d}X(t)=\nu(X(t))\mathrm{d}t+\sigma(X(t))\mathrm{d}B(t), \ X(0)=x$, where $B$ is a standard Brownian motion and $\sigma(y)$ is differentiable and non-zero inside the diffusion interval.
By defining the Lamperti transform $F(y) = \int_{y_0}^{y}\frac{1}{\sigma(z)}\mathrm{d}z$ for some $y_0$ from the diffusion interval of $X(t)$, and next considering $Y(t) \triangleq F(X(t))$, one can transform the process into one with unit diffusion coefficient.
Due to Itô's formula, the resulting process with unit diffusion coefficient will have a drift coefficient $\mu(y)$ given by the composition

$$\mu(y) = \left(\frac{\nu}{\sigma}-\frac{1}{2}\sigma'\right) \circ F^{-1}(y).$$

%
%

 The following result provides upper bounds for the first passage time densities of
 $\tau=\inf\{t >0: X(t) \geq g(t)\}$ and $\rho=\inf\{t >0: X(t) \leq g(t)\}$
 with $X$ a unit diffusion process.
 For a unit diffusion process with diffusion interval the whole real line and drift $\mu(\cdot)$,

 \begin{itemize}
 \item[1)] let $q(t,y,z)$ denote the transition density of the Brownian motion,
 \item[2)] let $\bar{g}(t)= \max _{0 \leq s \leq t}g(s)$, and $\underline{g}(t)=\min_{0 \leq s \leq t}g(s)$,
 \item[3)] define $G(y)=\int_{0}^{y}\mu(z)\mathrm{d}z$,
 \item[4)] define $M^*(t) = \essinf_{y \in [\underline{g}(t),\infty)}\left(\mu'(y)+\mu^2(y)\right)$, and, finally,
 \item[5)] define $M(t) = \essinf_{y \in (-\infty,\bar{g}(t)]}\left(\mu'(y)+\mu^2(y)\right)$.
 \end{itemize}

\begin{result}[Theorem~4.1 and Corollary~4.6 in \cite{downes2008boundary}]\label{bound-hitting-times-of-diffusion-processes}
Let $X(t)$ be a diffusion process satisfying $\mathrm{d}X(t)=\mu(X(t))\mathrm{d}t+\mathrm{d}B(t)$, $X(0)=x$, where $\mu$ is absolutely continuous and such that the SDE has a unique strong solution.
Let $g(t)$ be such that $g(0)<x$ and, for some $K >0$ and $T>0$,
\[
 g(t+h)-g(t) \leq Kh, \qquad 0 \leq t < t+h \leq T.
\]	
	Then, the first passage time $\tau$ of the process $X$ for the boundary $g(t)$, $t \in (0,T)$, has a density $p_{\rho},$ satisfying
	\[
	p_{\rho}(t) \leq B^{\ast}(g,t),
	\]
	where $B^{\ast}(g,t)$ is given by
	\[
	B^{\ast}(g,t)= \frac{1}{t}(x-g(t)+Kt)q(t,x,g(t))e^{G(g(t))-G(x)-\frac{t}{2} M^*(t)}.
	\]
	In addition, if $g(t)$ is such that $g(0)>x$ and, for some $K >0$ and $T>0$,
\[
-Kh \leq g(t+h)-g(t), \qquad 0 \leq t < t+h \leq T,
\]	
	then the first passage time $\tau$ of the process $X$ for the boundary $g(t)$, $t \in (0,T)$, has a density $p_{\tau},$ satisfying
	\[
	p_{\tau}(t) \leq B^{\circ}(g,t),
	\]
	where $B^{\circ}(g,t)$ is given by
	\[
	B^{\circ}(g,t)= \frac{1}{t}(g(t)+Kt-x)q(t,x,g(t))e^{G(g(t))-G(x)-\frac{t}{2} M(t)}.
	\]
\end{result}

%

\end{appendix}


\begin{funding}
Bazhba and Laeven were supported from the NWO under an NWO-Vici grant 2020--25.
Blanchet was supported from the NSF via grant DMS-EPSRC 2118199 and NSF-DMS 1915967, and from  AFOSR via grant MURI FA9550-20-1-0397.
\end{funding}

\bibliographystyle{imsart-number}
\bibliography{biblio-diffusion-process-project.bib}

\begin{thebibliography}{36}

\bibitem{asmussen2008applied}
\begin{bbook}[author]
\bauthor{\bsnm{Asmussen},~\bfnm{S{\o}ren}\binits{S.}}
(\byear{2008}).
\btitle{Applied Probability and Queues}.
\bseries{Stochastic Modelling and Applied Probability}
\bvolume{51}.
\bpublisher{Springer Science \& Business Media}.
\end{bbook}
\endbibitem

\bibitem{bazhba2020sample}
\begin{barticle}[author]
\bauthor{\bsnm{Bazhba},~\bfnm{Mihail}\binits{M.}},
  \bauthor{\bsnm{Blanchet},~\bfnm{Jose}\binits{J.}},
  \bauthor{\bsnm{Rhee},~\bfnm{Chang-Han}\binits{C.-H.}} \AND
  \bauthor{\bsnm{Zwart},~\bfnm{Bert}\binits{B.}}
(\byear{2020}).
\btitle{Sample-path large deviations for unbounded additive functionals of the
  reflected random walk}.
\bjournal{arXiv preprint arXiv:2003.14381}.
\end{barticle}
\endbibitem

\bibitem{bazhba2017sample}
\begin{barticle}[author]
\bauthor{\bsnm{Bazhba},~\bfnm{Mihail}\binits{M.}},
  \bauthor{\bsnm{Blanchet},~\bfnm{Jose}\binits{J.}},
  \bauthor{\bsnm{Rhee},~\bfnm{Chang-Han}\binits{C.-H.}} \AND
  \bauthor{\bsnm{Zwart},~\bfnm{Bert}\binits{B.}}
(\byear{2020}).
\btitle{Sample path large deviations for {L}\'{e}vy processes and random walks
  with {W}eibull increments}.
\bjournal{The Annals of Applied Probability}
\bvolume{30}
\bpages{2695--2739}.
\end{barticle}
\endbibitem

\bibitem{blanchet2013large}
\begin{barticle}[author]
\bauthor{\bsnm{Blanchet},~\bfnm{Jose}\binits{J.}},
  \bauthor{\bsnm{Glynn},~\bfnm{Peter}\binits{P.}} \AND
  \bauthor{\bsnm{Meyn},~\bfnm{Sean}\binits{S.}}
(\byear{2013}).
\btitle{Large deviations for the empirical mean of an $M/M/1$ queue}.
\bjournal{Queueing Systems}
\bvolume{73}
\bpages{425--446}.
\end{barticle}
\endbibitem

\bibitem{borovkov2003integral}
\begin{barticle}[author]
\bauthor{\bsnm{Borovkov},~\bfnm{Aleksandr~Alekseevi{\v{c}}}\binits{A.~A.}},
  \bauthor{\bsnm{Boxma},~\bfnm{Onno~Johan}\binits{O.~J.}} \AND
  \bauthor{\bsnm{Palmowski},~\bfnm{Zbigniew}\binits{Z.}}
(\byear{2003}).
\btitle{On the integral of the workload process of the single server queue}.
\bjournal{Journal of Applied Probability}
\bvolume{40}
\bpages{200--225}.
\end{barticle}
\endbibitem

\bibitem{bryc1997large}
\begin{barticle}[author]
\bauthor{\bsnm{Bryc},~\bfnm{W{\l}odzimierz}\binits{W.}} \AND
  \bauthor{\bsnm{Dembo},~\bfnm{Amir}\binits{A.}}
(\byear{1997}).
\btitle{Large deviations for quadratic functionals of Gaussian processes}.
\bjournal{Journal of Theoretical Probability}
\bvolume{10}
\bpages{307--332}.
\end{barticle}
\endbibitem

\bibitem{MR3185356}
\begin{barticle}[author]
\bauthor{\bparticle{de} \bsnm{Acosta},~\bfnm{Alejandro~D.}\binits{A.~D.}} \AND
  \bauthor{\bsnm{Ney},~\bfnm{Peter}\binits{P.}}
(\byear{2014}).
\btitle{Large deviations for additive functionals of {M}arkov chains}.
\bjournal{Memoirs of the American Mathematical Society}
\bvolume{228}
\bpages{vi+108}.
\end{barticle}
\endbibitem

\bibitem{dembo1997uniform}
\begin{barticle}[author]
\bauthor{\bsnm{Dembo},~\bfnm{Amir}\binits{A.}} \AND
  \bauthor{\bsnm{Zajic},~\bfnm{Tim}\binits{T.}}
(\byear{1997}).
\btitle{Uniform large and moderate deviations for functional empirical
  processes}.
\bjournal{Stochastic Processes and Their Applications}
\bvolume{67}
\bpages{195--211}.
\end{barticle}
\endbibitem

\bibitem{dembo2010large}
\begin{bbook}[author]
\bauthor{\bsnm{Dembo},~\bfnm{Amir}\binits{A.}} \AND
  \bauthor{\bsnm{Zeitouni},~\bfnm{Ofer}\binits{O.}}
(\byear{2010}).
\btitle{Large Deviations Techniques and Applications}.
\bpublisher{Springer-Verlag, Berlin}.
\end{bbook}
\endbibitem

\bibitem{den2019properties}
\begin{barticle}[author]
\bauthor{\bparticle{den} \bsnm{Hollander},~\bfnm{Frank}\binits{F.}},
  \bauthor{\bsnm{Majumdar},~\bfnm{Satya~N.}\binits{S.~N.}},
  \bauthor{\bsnm{Meylahn},~\bfnm{Janusz~M.}\binits{J.~M.}} \AND
  \bauthor{\bsnm{Touchette},~\bfnm{Hugo}\binits{H.}}
(\byear{2019}).
\btitle{Properties of additive functionals of {B}rownian motion with
  resetting}.
\bjournal{Journal of Physics A: Mathematical and Theoretical}
\bvolume{52}
\bpages{175001, 24}.
\end{barticle}
\endbibitem

\bibitem{MR386024}
\begin{barticle}[author]
\bauthor{\bsnm{Donsker},~\bfnm{M.~D.}\binits{M.~D.}} \AND
  \bauthor{\bsnm{Varadhan},~\bfnm{S.~R.~S.}\binits{S.~R.~S.}}
(\byear{1975}).
\btitle{Asymptotic evaluation of certain {M}arkov process expectations for
  large time. {I}. {II}}.
\bjournal{Communications on Pure and Applied Mathematics}
\bvolume{28}
\bpages{1--47; ibid. \textbf{28} (1975), 279--301}.
\end{barticle}
\endbibitem

\bibitem{MR428471}
\begin{barticle}[author]
\bauthor{\bsnm{Donsker},~\bfnm{M.~D.}\binits{M.~D.}} \AND
  \bauthor{\bsnm{Varadhan},~\bfnm{S.~R.~S.}\binits{S.~R.~S.}}
(\byear{1976}).
\btitle{Asymptotic evaluation of certain {M}arkov process expectations for
  large time. {III}}.
\bjournal{Communications on Pure and Applied Mathematics}
\bvolume{29}
\bpages{389--461}.
\end{barticle}
\endbibitem

\bibitem{MR690656}
\begin{barticle}[author]
\bauthor{\bsnm{Donsker},~\bfnm{M.~D.}\binits{M.~D.}} \AND
  \bauthor{\bsnm{Varadhan},~\bfnm{S.~R.~S.}\binits{S.~R.~S.}}
(\byear{1983}).
\btitle{Asymptotic evaluation of certain {M}arkov process expectations for
  large time. {IV}}.
\bjournal{Communications on Pure and Applied Mathematics}
\bvolume{36}
\bpages{183--212}.
\end{barticle}
\endbibitem

\bibitem{downes2008boundary}
\begin{bbook}[author]
\bauthor{\bsnm{Downes},~\bfnm{Andrew~Nicholas}\binits{A.~N.}}
(\byear{2008}).
\btitle{Boundary Crossing Probabilities for Diffusion Processes and Related
  Problems}.
\bpublisher{Ph.D. thesis, The University of Melbourne}.
\end{bbook}
\endbibitem

\bibitem{duffy2010most}
\begin{barticle}[author]
\bauthor{\bsnm{Duffy},~\bfnm{Ken~R}\binits{K.~R.}} \AND
  \bauthor{\bsnm{Meyn},~\bfnm{Sean~P}\binits{S.~P.}}
(\byear{2010}).
\btitle{Most likely paths to error when estimating the mean of a reflected
  random walk}.
\bjournal{Performance Evaluation}
\bvolume{67}
\bpages{1290--1303}.
\end{barticle}
\endbibitem

\bibitem{duffy2014large}
\begin{barticle}[author]
\bauthor{\bsnm{Duffy},~\bfnm{Ken~R}\binits{K.~R.}} \AND
  \bauthor{\bsnm{Meyn},~\bfnm{Sean~P}\binits{S.~P.}}
(\byear{2014}).
\btitle{Large deviation asymptotics for busy periods}.
\bjournal{Stochastic Systems}
\bvolume{4}
\bpages{300--319}.
\end{barticle}
\endbibitem

\bibitem{gantert2014large}
\begin{barticle}[author]
\bauthor{\bsnm{Gantert},~\bfnm{Nina}\binits{N.}},
  \bauthor{\bsnm{Ramanan},~\bfnm{Kavita}\binits{K.}} \AND
  \bauthor{\bsnm{Rembart},~\bfnm{Franz}\binits{F.}}
(\byear{2014}).
\btitle{Large deviations for weighted sums of stretched exponential random
  variables}.
\bjournal{Electronic Communications in Probability}
\bvolume{19}
\bpages{1--14}.
\end{barticle}
\endbibitem

\bibitem{guillemin1998area}
\begin{barticle}[author]
\bauthor{\bsnm{Guillemin},~\bfnm{Fabrice}\binits{F.}} \AND
  \bauthor{\bsnm{Pinchon},~\bfnm{Didier}\binits{D.}}
(\byear{1998}).
\btitle{On the area swept under the occupation process of an M/M/1 queue in a
  busy period}.
\bjournal{Queueing Systems}
\bvolume{29}
\bpages{383--398}.
\end{barticle}
\endbibitem

\bibitem{gulinskii1994large}
\begin{barticle}[author]
\bauthor{\bsnm{Gulinskii},~\bfnm{OV}\binits{O.}},
  \bauthor{\bsnm{Lipster},~\bfnm{Robert~S}\binits{R.~S.}} \AND
  \bauthor{\bsnm{Lototskii},~\bfnm{SV}\binits{S.}}
(\byear{1994}).
\btitle{Large deviations for unbounded additive functionals of a Markov process
  with discrete time (noncompact case)}.
\bjournal{Journal of Applied Mathematics and Stochastic Analysis}
\bvolume{7}
\bpages{423--436}.
\end{barticle}
\endbibitem

\bibitem{10.1214/aop/1176990736}
\begin{barticle}[author]
\bauthor{\bsnm{Jain},~\bfnm{Naresh~C.}\binits{N.~C.}}
(\byear{1990}).
\btitle{{Large deviation lower bounds for additive functionals of Markov
  processes}}.
\bjournal{The Annals of Probability}
\bvolume{18}
\bpages{1071--1098}.
\end{barticle}
\endbibitem

\bibitem{karatzas1998brownian}
\begin{bincollection}[author]
\bauthor{\bsnm{Karatzas},~\bfnm{Ioannis}\binits{I.}} \AND
  \bauthor{\bsnm{Shreve},~\bfnm{Steven~E}\binits{S.~E.}}
(\byear{1998}).
\btitle{Brownian motion}.
In \bbooktitle{Brownian Motion and Stochastic Calculus}
\bpages{47--127}.
\bpublisher{Springer}.
\end{bincollection}
\endbibitem

\bibitem{klebaner2012introduction}
\begin{bbook}[author]
\bauthor{\bsnm{Klebaner},~\bfnm{Fima~C}\binits{F.~C.}}
(\byear{2012}).
\btitle{Introduction to Stochastic Calculus with Applications}.
\bpublisher{World Scientific Publishing Company}.
\end{bbook}
\endbibitem

\bibitem{kontoyiannis2005large}
\begin{barticle}[author]
\bauthor{\bsnm{Kontoyiannis},~\bfnm{Ioannis}\binits{I.}} \AND
  \bauthor{\bsnm{Meyn},~\bfnm{Sean}\binits{S.}}
(\byear{2005}).
\btitle{Large deviations asymptotics and the spectral theory of
  multiplicatively regular Markov processes}.
\bjournal{Electronic Journal of Probability}
\bvolume{10}
\bpages{61--123}.
\end{barticle}
\endbibitem

\bibitem{kulik2005tail}
\begin{barticle}[author]
\bauthor{\bsnm{Kulik},~\bfnm{Rafa{\l}}\binits{R.}} \AND
  \bauthor{\bsnm{Palmowski},~\bfnm{Zbigniew}\binits{Z.}}
(\byear{2005}).
\btitle{Tail behaviour of the area under the queue length process of the single
  server queue with regularly varying service times}.
\bjournal{Queueing Systems}
\bvolume{50}
\bpages{299--323}.
\end{barticle}
\endbibitem

\bibitem{meerson2019anomalous}
\begin{barticle}[author]
\bauthor{\bsnm{Meerson},~\bfnm{Baruch}\binits{B.}}
(\byear{2019}).
\btitle{Anomalous scaling of dynamical large deviations of stationary Gaussian
  processes}.
\bjournal{Physical Review E}
\bvolume{100}
\bpages{042135}.
\end{barticle}
\endbibitem

\bibitem{ney1987markov}
\begin{barticle}[author]
\bauthor{\bsnm{Ney},~\bfnm{Peter}\binits{P.}} \AND
  \bauthor{\bsnm{Nummelin},~\bfnm{Esa}\binits{E.}}
(\byear{1987}).
\btitle{Markov additive processes II. Large deviations}.
\bjournal{The Annals of Probability}
\bvolume{15}
\bpages{593--609}.
\end{barticle}
\endbibitem

\bibitem{nickelsen2018anomalous}
\begin{barticle}[author]
\bauthor{\bsnm{Nickelsen},~\bfnm{Daniel}\binits{D.}} \AND
  \bauthor{\bsnm{Touchette},~\bfnm{Hugo}\binits{H.}}
(\byear{2018}).
\btitle{Anomalous scaling of dynamical large deviations}.
\bjournal{Physical Review Letters}
\bvolume{121}
\bpages{090602}.
\end{barticle}
\endbibitem

\bibitem{nickelsen2022noise}
\begin{barticle}[author]
\bauthor{\bsnm{Nickelsen},~\bfnm{Daniel}\binits{D.}} \AND
  \bauthor{\bsnm{Touchette},~\bfnm{Hugo}\binits{H.}}
(\byear{2022}).
\btitle{Noise correction of large deviations with anomalous scaling}.
\bjournal{arXiv preprint arXiv:2202.07348}.
\end{barticle}
\endbibitem

\bibitem{smith2022anomalous}
\begin{barticle}[author]
\bauthor{\bsnm{Smith},~\bfnm{Naftali~R.}\binits{N.~R.}}
(\byear{2022}).
\btitle{Anomalous scaling and first-order dynamical phase transition in large
  deviations of the Ornstein-Uhlenbeck process}.
\bjournal{Physical Review E}
\bvolume{105}
\bpages{014120}.
\end{barticle}
\endbibitem

\bibitem{MR1730651}
\begin{barticle}[author]
\bauthor{\bsnm{Stramer},~\bfnm{O.}\binits{O.}} \AND
  \bauthor{\bsnm{Tweedie},~\bfnm{R.~L.}\binits{R.~L.}}
(\byear{1999}).
\btitle{Langevin-type models. {I}. {D}iffusions with given stationary
  distributions and their discretizations}.
\bjournal{Methodology and Computing in Applied Probability}
\bvolume{1}
\bpages{283--306}.
\end{barticle}
\endbibitem

\bibitem{stroock1997multidimensional}
\begin{bbook}[author]
\bauthor{\bsnm{Stroock},~\bfnm{D.~W.}\binits{D.~W.}} \AND
  \bauthor{\bsnm{Varadhan},~\bfnm{S.~R.~S.}\binits{S.~R.~S.}}
(\byear{2006}).
\btitle{Multidimensional Diffusion Processes}.
\bseries{Grundlehren der mathematischen Wissenschaften}.
\bpublisher{Springer Berlin Heidelberg}.
\end{bbook}
\endbibitem

\bibitem{MR2391248}
\begin{barticle}[author]
\bauthor{\bsnm{Takeda},~\bfnm{Masayoshi}\binits{M.}}
(\byear{2008}).
\btitle{Large deviations for additive functionals of symmetric stable
  processes}.
\bjournal{Journal of Theoretical Probability}
\bvolume{21}
\bpages{336--355}.
\end{barticle}
\endbibitem

\bibitem{veretennikov1994large}
\begin{barticle}[author]
\bauthor{\bsnm{Veretennikov},~\bfnm{A~Yu}\binits{A.~Y.}}
(\byear{1994}).
\btitle{On large deviations for additive functionals of Markov processes I}.
\bjournal{Theory of Probability \& Its Applications}
\bvolume{38}
\bpages{706--719}.
\end{barticle}
\endbibitem

\bibitem{wu1994large}
\begin{barticle}[author]
\bauthor{\bsnm{Wu},~\bfnm{Liming}\binits{L.}}
(\byear{1994}).
\btitle{Large deviations, moderate deviations and LIL for empirical processes}.
\bjournal{The Annals of Probability}
\bvolume{22}
\bpages{17--27}.
\end{barticle}
\endbibitem

\bibitem{wuyiao2008}
\begin{barticle}[author]
\bauthor{\bsnm{Wu},~\bfnm{Liming}\binits{L.}} \AND
  \bauthor{\bsnm{Yao},~\bfnm{Nian}\binits{N.}}
(\byear{2008}).
\btitle{{Large deviation principles for Markov processes via Phi-Sobolev
  inequalities}}.
\bjournal{Electronic Communications in Probability}
\bvolume{13}
\bpages{10 -- 23}.
\end{barticle}
\endbibitem

\bibitem{yamada1973comparison}
\begin{barticle}[author]
\bauthor{\bsnm{Yamada},~\bfnm{Toshio}\binits{T.}}
(\byear{1973}).
\btitle{On a comparison theorem for solutions of stochastic differential
  equations and its applications}.
\bjournal{Journal of Mathematics of Kyoto University}
\bvolume{13}
\bpages{497--512}.
\end{barticle}
\endbibitem

\end{thebibliography}
\end{document}